\documentclass[12pt]{amsart}
\usepackage{amsthm,amsmath,amssymb,graphicx}
\usepackage{url}
\usepackage[top=25mm, bottom=25mm, left=30mm, right=30mm]{geometry}
\usepackage{fourier}
\usepackage[colorlinks=true,citecolor=black,linkcolor=black,urlcolor=blue]{hyperref}
\usepackage{enumerate}

\theoremstyle{plain}
\newtheorem{defn}{Definition}[section]

\newtheorem{lemma}[defn]{Lemma}
\newtheorem{theorem}[defn]{Theorem}
\newtheorem{corollary}[defn]{Corollary}

\DeclareMathOperator{\im}{im}
\DeclareMathOperator{\Id}{Id}

\begin{document} 
\title{Counting integer points in multi-index transportation polytopes} 

\author{David Benson-Putnins} 
\address{Department of Mathematics \\ University of Michigan \\
530 Church Street \\ Ann Arbor, Michigan \\ USA} 
\email{dputnins@umich.edu} 


\thanks{\textbf{Appeared in OJAC 2014. See \texttt{analytic-combinatorics.org.}}}
\thanks{\textbf{This article is licensed under a Creative Commons Attribution 4.0 International license.}}

\begin{abstract}
We expand on a result of Barvinok and Hartigan to derive asymptotic formulas for the number of integer and binary integer points in a wide class of multi-index $k_1\times k_2\times \ldots \times k_{\nu}$ transportation polytopes.  A simple closed form approximation is given as the $k_j$s go to infinity.
\end{abstract}

\date{\today}
\maketitle
\section{Introduction} \label{s:Intro}
A \emph{$\nu$-index transportation polytope} is a set of $k_1\times \ldots \times k_{\nu}$ arrays of non-negative numbers with fixed hyper hypercube arrays of non-negative numbers of the form
\[ (\xi_{m_1,\ldots ,m_{\nu}})_{m_1,\ldots,m_{\nu} = 1}^{k_1,\ldots ,k_{\nu}} \]
satisfying the following relations: Fix some arbitrary $j$ with $1\leq j \leq \nu$, and some arbitrary $m_j$ with $1\leq m_j \leq k_j$.  Then there are constants $S^j_{m_j}$ for each such $j$ and $m_j$ such that
\[ \sum_{ m_1,\ldots ,m_{j-1},m_{j+1},\ldots ,m_{\nu}} \xi_{m_1, m_2, \ldots, m_{\nu}} = S^{j}_{m_j}. \]
For such a $j$ and $m_j$ we call  $S^j_{m_j}$ the $m_j$th \emph{margin} in the $j$th direction.  In the literature this is often referred to as a multi-index transportation polytope with \emph{fixed} $1$-\emph{margins}, for example in  \cite{exampleofmargins}, or a \emph{planar} transportation polytope in the 3-index case, for example in \cite{exampleofplanar}.

\

Counting the number of binary integer points in a $\nu$-index transportation polytope is a special case of counting $\nu$ uniform, $\nu$-partite hypergraphs.  The vertices of the $j$th partition are labeled $1$ through $k_j$, and an entry in the $m_1,m_2,\ldots, m_{\nu}$ position says there exists an edge connecting vertices $m_1,m_2,\ldots,m_{\nu}$.  Work has gone into counting asymptotically the number of hypergraphs of certain forms, see \cite{hypergraphs}.

\

Counting the number of integer and binary integer points in $3$-way contingency tables has applications in algebraic combinatorics.  The \emph{Kronecker coefficients} $g(\lambda,\mu,\nu)$ for partitions $\lambda$, $\mu$ and $\nu$ of some integer $n$ are defined by the identity
\[ \chi^{\lambda}\otimes \chi^{\mu} = \sum_{\nu} g(\lambda, \mu,\nu) \chi^{\nu} \]
where $\chi^{\alpha}$ is the irreducible representation of $S_n$ indexed by partition $\alpha$.  It is known that the values of $g(\lambda, \mu,\nu)$ are non-negative, but a combinatorial interpretation or simple counting formula is not known.  In \cite{kroneckerBound} it is shown that the the value of $g(\lambda, \mu, \nu)$ can be bounded from above by the number of integer points of a $3$-way contingency table whose margins are given by $\lambda$, $\mu$ and $\nu$. It is also shown that $g(\lambda, \mu, \nu)$ is bounded from above  by the number of binary integer points of a $3$-way contingency table whose margins are given by $\lambda'$, $\mu$ and $\nu$, where $\lambda'$ is the conjugate partition of $\lambda$.  In \cite{kroneckerExact} it is shown that $g(\lambda,\mu,\nu)$ can be calculated exactly in terms of the number of integer points in $3$-way contingency tables of various margins.

\

In statistics, points in a $\nu$-index transportation polytopes tables are constructed from a given dataset in the following way: $N$ objects have $\nu$ qualities divided into $k_1$ categories for the first quality, $k_2$ categories for the second, through $k_{\nu}$ categories for the last.  The entry $x_{m_1,\ldots,m_{\nu}}$ is the number of objects that have quality $1$ fall into category $m_1$, quality $2$ into category $m_2$, through quality $\nu$ in category $m_{\nu}$.  Estimating the number of integer points contained within the corresponding transportation polytope is critical for tests of significance in the distribution of contingency tables and interpretation of those results.  See \cite{contingencytabletests} for an exposition in the $\nu=2$ case. 

\

The main results in this paper are asymptotic formulas that approximately count the number of integer points and binary integer points in a wide class of $\nu$-index transportation polytopes for $\nu \geq 3$.  Much work has been done in calculating asymptotic formulas for integer points in special cases.  Examples include two-directional  transportation polytopes, also known as contingency tables, in the sparse case in \cite{sparse}, and in the case of all equal margins in \cite{constantrowandcolumn}.  An asymptotic formula for the number of integer points in certain "smooth" - or close in a certain technical sense to the case of all equal margins - two-directional multi-index transportation polytopes has been calculated in \cite{2dcase}.  Formulas for the volume and number of integer points and binary points for smooth multi-index transportation polytope of five or more directions was found in in \cite{entropy}.  It was not previously known that smooth three and four directional transportation polytopes allowed the same asymptotic formula.  In addition, the asymptotic error for the case of having five or more directions is improved over that calculated in \cite{entropy}.  We will combine the approaches of these last two papers, along with improved estimates on the variance of certain Gaussian random variables to achieve the result.   

\

The layout of this paper is: the remainder of Section ~\ref{s:Intro} states the two main theorems, and discusses future potential work related to them.  In the proof of both theorems we rely heavily on the eigenspace of quadratic forms of a certain type.  In Sections ~\ref{s:Eigenspace}, ~\ref{s:Variance}, ~\ref{s:Correlations} and ~\ref{s:ThirdDegreeTerm}, we calculate the eigenspace of these quadratic forms, and prove several lemmas and theorems common to the proofs of both main theorems.  The main theorems are then proven in Sections ~\ref{s:IntegerPoints} and ~\ref{s:BinaryPoints}.
\subsection{The Polytope Constraints} \label{ss:ConstraintMatrix} A set $P \subset \mathbb{R}^n$ is called a polyhedron - and a polytope, if it is bounded - if it can be defined as
\[ P = \left\{ x = (\xi_1,\ldots,\xi_n)\ :\ Ax = b\ \ \ \text{and}\ \ \ \xi_j \geq 0 \text{  for all } j \right\} \]
for some $A$ a $d\times n$ matrix of real numbers, and $b\in \mathbb{R}^d$.
In this case the columns of $A$ will be denoted $a_1,\ldots ,a_n$.  For the $\nu$-index transportation polytope defined earlier, we write a point in our hypercube array as
\[ (\xi_{11\ldots 1}, \xi_{11\ldots 2},\ldots ,\xi_{11\ldots k_1}, \xi_{11\ldots 121}, \ldots , \xi_{k_1 k_2\ldots k_{\nu} } )\]
with the coordinate $\xi_{m_1 \ldots m_{\nu}}$ being the coordinate lying in the $m_j$th margin of the $j$th direction. The transportation polytope then fits the above definition with $ n = k_1 k_2\ldots k_{\nu}$, and each $a_{m_1\ldots m_{\nu}}$ being a vector of length $k_1+k_2+\ldots +k_{\nu}$ that has all $0$s, except for a $1$ in positions $m_1$, $k_1+m_2$, $k_1+k_2+m_3$,..., and $k_1+\ldots +k_{\nu-1}+m_{\nu}$.  In this case the entry of $b$ in position $k_1+\ldots +k_{j-1} + m_j$ is $S^j_{m_j}$ for each $j$ and $m_j$.

\

It is important to note that the constraint matrix $A$ of a multi-index transportation polytope does not have full rank.  This is easily seen by observing that for each $j$, the sum
\[ \sum_{m_j=1}^{k_j} S^j_{m_j} \]
must be the same value, as it gives the sum of all entries in the hypercube array.  This is the only linear dependency amongst the constraints, and a basis of the constraints consists of removing the constraint on the $k_j$th margin in the $j$th direction for $j=2,\ldots ,\nu$.  If $\mathcal{L} \subset \mathbb{R}^{k_1+\ldots +k_{\nu}}$ is the subspace
\begin{equation} \label{eq:DefinitionOfL}\mathcal{L} = \left\{(\xi_1,\ldots ,\xi_{k_1+\ldots +k_{\nu}})\ :\ \xi_{k_1+\ldots +k_j} = 0\ \ \ \text{for all}\ \ \ 2\leq j \leq \nu \right\}, \end{equation}
and $Q:\mathbb{R}^{k_1+\ldots +k_{\nu}}\to \mathbb{R}^{k_1+\ldots +k_{\nu}}$ is the orthogonal projection onto $\mathcal{L}$, then $QA$ is a full rank linear transformation from $\mathbb{R}^{k_1+\ldots+k_{\nu}} \to \mathcal{L}$ and the system of constraints $QAx = Qb$ is equivalent to selecting a linearly independent set of constraints for $P$.

\subsection{Quadratic Forms and Inner Products}\label{ss:QuadraticFormBasics}
Recall if $q(t)$ is a positive semidefinite quadratic form on $\mathbb{R}^d$, then there exists a positive semidefinite symmetric matrix $B$ such that $q(t) = \frac{1}{2}\left<t, Bt \right>$.  We define the eigenvalues and eigenvectors of $q(t)$ to simply be the eigenvalues and eigenvectors of $B$.  If $\mathcal{V} \subset \mathbb{R}^d$ is a linear subspace and $Q:\mathbb{R}^{d} \to \mathbb{R}^d$ the orthogonal projection onto $\mathcal{V}$, then $q|_\mathcal{V}(t)$ will denote the quadratic form $\frac{1}{2}\left<t, QBQ t \right>$.  For $t\in \mathcal{V}$ this conforms with the original definition, but we will occasionally decompose $t$ into vectors not contained in $\mathcal{V}$ which will make this definition convenient.

\

If $B$ is positive semidefinite symmetric $d\times d$ matrix, then $QBQ$ is a positive semidefinite symmetric  $d\times d$ matrix whose kernel includes $\mathcal{V}^{\perp}$. Therefore there exists a basis of orthogonal eigenvectors that all lie in $\mathcal{V}$ or $\mathcal{V}^{\perp}$.  By $\det(q)$ we mean the product of the eigenvalues of $B$, and by $\det(q|_{\mathcal{V}})$ we mean the product of the eigenvalues of the eigenvectors of $QBQ$ that lie in $\mathcal{V}$.

\

Lastly, we recall that if $q$ is a positive definite quadratic form on some subspace $\mathcal{V}\subset \mathbb{R}^n$, then
\begin{equation}\label{eq:IntegralOfGaussian}\int_{\mathcal{V}} e^{-q(t)} dt = \frac{(2\pi)^{\dim(\mathcal{V})/2}}{\sqrt{\det(q|_{\mathcal{V}})}}. \end{equation}
\subsection{Maximum Entropy in the Counting Problem}\label{ss:maxEntropyHeuristic}
In many counting and volume measurement problems, the problem is reduced to calculating an integral.  Examples include \cite{2dcase} and \cite{entropy} in counting integer points of general polytopes.  In \cite{degSeq}, the number of graphs satisfying certain conditions on the degrees of its vertices is counted in a similar manner. 

\ 

The principle that allows the construction of the integral is inspired by the standard 'Monte Carlo' or random sampling method.  To count the number of integer points in a polytope $P\subset \mathbb{R}^n_+$ defined by the system of equations $Ax = b$, we first construct a random variable $X$ that takes values in $\mathbb{Z}^n_+$, for which $\mathbf{E}X\in P$.  We then express $\left|P\cap \mathbb{Z}^n \right|$ as a function of $\mathbf{Pr}(X\in P)$.  Rather than a numerical sampling to estimate $\mathbf{Pr}(X\in P)$, we use $X\in P$ if and only if $AX = b$.  If $A$ is $d\times n$ with $n \gg d$, and each row of $A$ has sufficiently many nonzero entries, and each entry of $X$ is picked independently, then the entries of $AX$ are approximately Gaussian by the Central Limit Theorem.  The integrand of the integral we use to estimate the number of integer points is simply $e^{-q(t)}$, where $q(t)$ is a certain quadratic form that we construct later. 

\

It turns out that a useful choice of $X$ is the random variable of maximum entropy whose expected value lies in $P$ that takes the relevant values.  In the integer point case the entries of $X$ are independent geometric random variables, and in the binary integer point case the entries of $X$ are independent Bernoulli random variables.  In Sections ~\ref{ss:IntegerPointEntropy} and ~\ref{ss:BinaryPointEntropy}, we cite several lemmas and theorems of Barvinok and Hartigan that describe the choice of a random variable $X$ that is appropriate for counting integer or binary integer points, and how to construct the probability mass function of the distribution explicitly, but the focus of this paper will be on the application of these theorems to the specific example of multi-index transportation polytopes.  See \cite{entropy} for more details on the general case.
\subsection{Counting Integer Points of Transportation Polytopes}
In this section we state the main theorem estimating the number of integer points in a multi-index transportation polytope.
\begin{theorem}\label{IntegerPoints}Let $P$ be a $k_1\times \ldots \times k_{\nu}$ multi-index transportation polytope with $\nu \geq 3$ defined as in Section ~\ref{ss:ConstraintMatrix} by the overdetermined linear equations $Ax = b$ with $\mathcal{L}$ the subspace defining a linearly independent subset of equations, and let $n = k_1\times \ldots \times k_{\nu}$.  Let $z = (\zeta_1,\ldots,\zeta_n)$ be the unique point in $P$ on which the strictly concave function
\[ g(x) = \sum_{j=1}^{n} (\xi_j+1)\ln(\xi_j+1) - \xi_j \ln(\xi_j)\ \ \ \text{for}\ \ \ x = (\xi_1,\ldots,\xi_n)\]
attains its maximum value.  Let $D$ be the matrix whose columns are $(\zeta_j+\zeta_j^2)^{1/2}a_j$, where $a_j$ are the columns of $A$, and let $q(t) = \frac{1}{2}\left<Dt,Dt\right>$.  Suppose there exist numbers $0 < \omega <1 $, along with $k>0$, and $R > r > 0$ such that the following inequalities hold:
\[ \omega k \leq k_j \leq k\qquad \text{for}\qquad j=1,\ldots,\nu,\qquad and \]
\[ r \leq \zeta_j^2+\zeta_j \leq R\qquad \text{for}\qquad j=1,\ldots,n, \]
along with the inequalities $\omega k \geq 2$, and $ R > 1$.  If $k$ is large enough to satisfy the following inequalities:
\[\frac{8 \pi^2 2^{\nu} \nu^2}{\omega^{\nu}\ln(1+\frac{2}{5}\pi^2 r)}\left(\frac{1}{2} \nu^2 k \ln(k) + \frac{1}{2} \nu k \ln(R) \right) k^{-\nu+1} \leq \frac{1}{4\nu^2R},\ \ \ \text{and}\]
\[\frac{64 \pi^2 2^{\nu} \nu^6 R^2}{\omega^{\nu}r} \ln(k) k^{-\nu+2} \leq 3/4,\]

then $\left|P\cap \mathbb{Z}^d \right|$ is approximated by
\[ \frac{e^{g(z)}}{(2\pi)^{(k_1+\ldots+k_{\nu}-\nu+1)/2}} \det(q|_{\mathcal{L}})^{-1/2}\]
to within relative error
\[ \Gamma k^{-\nu+2.5}\]
for some constant $\Gamma = \Gamma(R,r,\omega,\nu)$.  In particular, if $r$, $R$, $\omega$ and $\nu$ are fixed, there exists $N = N(r,R,\omega,\nu)$ such that for all $k \geq N$, we have
\[\Gamma = \frac{256 R^2 \pi^4 4^{\nu}\nu^8}{\omega^{2\nu}}.\]
\end{theorem}
The conditions of Theorem ~\ref{IntegerPoints} essentially say that $P$ looks similar to the most symmetric case possible.  $P$ is called a \emph{polystochastic tensor} if $k_1 = k_2 = \ldots = k_{\nu}$ and every entry of $b$ is equal to $k^{\nu-1}$.  In this case we can take $k = k_j$ for all $j=1,\ldots, \nu$, and $\omega = 1$.  Furthermore, by the symmetry of the problem, we get $\zeta_j = 1$ for $j=1,\ldots, n$.  The value of $\omega$ measures how far from a hypercube the shape of the polytope's arrays are. The values $R/r$ essentially measure how far from equal the entries of $b$ are, and the magnitude of $r$ (or $R$) is a measure of how large the entries of $b$ are.  

\

The assumption that $R> 1$ is trivial, as $R$ is simply an upper bound on the values of $\zeta_{m_1,\ldots,m_{\nu}}$ and can be chosen to be larger if needed.  If any of the $k_j$s are equal to $1$, then every entry of $P$ is uniquely determined and there is nothing to count, so $\omega k \geq 2$ is also a trivial assumption.  The two non-trivial assumptions on how large $k$ is are generated by the specific proof we use.   Informally, they say if $R$ is too large compared to $k$, the theorem is not valid.  Fixing $R/r$ and letting $r,R$ go to infinity is equivalent to letting the margin sums go to infinity.  In this case the number of integer points well-approximates the volume of $P$ \cite{IntegerPointsToVolume}.  As we will discuss in Section ~\ref{ss:FutureWork} this restriction is likely artificial. Under the heuristic described in ~\ref{ss:maxEntropyHeuristic}, we would expect the problem of estimating the number of integer points to become easier as the margins go to infinity.

\

To prove Theorem ~\ref{IntegerPoints}, we show that the number of integer points in $P$ can be expressed using 
\[ \int_{\Pi} F(t) dt, \]
where $\Pi \subset \mathcal{L}$ for $\mathcal{L}$ as in ~\eqref{eq:DefinitionOfL} is a cube centered at the origin whose sides have length $2\pi$, and $F(t)$ is a function that will be defined later.  We then split $\mathcal{L}$ into three regions $X_1$, $X_2$, and $X_3$. We show that
\[ \int_{(X_2\cup X_3)\cap \Pi}  |F(t)| dt\ \ \ \text{and}\ \ \ \int_{X_2\cup X_3} e^{-q(t)}dt \ll \int_{\mathcal{L}} e^{-q(t)} dt, \]
where $q(t)$ is the quadratic form constructed in Theorem ~\ref{IntegerPoints}, and that
\[ \int_{X_1} F(t) dt \approx \int_{X_1} e^{-q(t)}dt \approx \int_{\mathcal{L}} e^{-q(t)} dt.\]
To facilitate these calculations we will require several results about the probability distribution whose density is proportional to $e^{-q(t)}$ on $\mathcal{L}$.  Sections ~\ref{s:Eigenspace}, ~\ref{s:Correlations} and ~\ref{s:ThirdDegreeTerm} will contain these, and the proof of Theorem ~\ref{IntegerPoints} will take place in Section ~\ref{s:IntegerPoints}.
\subsection{Counting Binary Points in Transportation Polytopes}
In this section we state the main theorem estimating the number of binary integer points in a multi-index transportation polytope.
\begin{theorem}\label{BinaryPoints}Let $P$ be a $k_1\times \ldots \times k_{\nu}$ transportation polytope defined by the overdetermined linear equations $Ax = b$ as described in Section ~\ref{ss:ConstraintMatrix}, with $\mathcal{L}$ the subspace defining a linearly independent set of equations, and let $n = k_1\times \ldots \times k_{\nu}$.  Let $z = (\zeta_1,\ldots,\zeta_n)$ be the unique point in $P\cap[0,1]^n$ on which the strictly concave function
\[ g(x) = \sum_{j=1}^{n} \xi_j\ln\frac{1}{\xi_j} + (1-\xi_j) \ln\frac{1}{1-\xi_j}\ \ \ \text{for}\ \ \ x = (\xi_1,\ldots,\xi_n)\]
attains its maximum value.  Let $D$ be the matrix whose columns are $(\zeta_j-\zeta_j^2)^{1/2}a_j$, where $a_j$ are the columns of $A$, and let $q(t) = \frac{1}{2}\left<Dt,Dt\right>$.  Suppose there exist numbers $0 < \omega < 1$, along with $k>0$ and $ r > 0$ such that
\[ \omega k \leq k_j \leq k\ \ \ \text{for}\ \ \ j=1,\ldots,\nu,\ \ \ and \]
\[ r \leq \zeta_j-\zeta_j^2 \ \ \ \text{for}\ \ \ j=1,\ldots,n, \]
along with $\omega k \geq 2$.  If $k$ is large enough so that
\[ \frac{10 \nu^{4} 2^{\nu}}{r\omega^{\nu}} \ln(k) k^{-\nu+2} \leq \frac{1}{4\nu^2}\ \ \ \text{and}\]
\[ \frac{20 \nu^6 2^{\nu}}{r^2 \omega^{\nu}} \ln(k) k^{-\nu+2} \leq 3/4,\]
then $\left|P\cap \{0,1\}^d \right|$ is approximated by
\[ \frac{e^{g(z)}}{(2\pi)^{(k_1+\ldots+k_{\nu}-\nu+1)/2}} \det(q|_{\mathcal{L}})^{-1/2}\]
to within relative error
\[ \Gamma k^{-\nu+2.5}\]
for some constant $\Gamma = \Gamma(r,\omega,\nu)$. There exists some constant $N = N(r,\omega, \nu)$ such that if $k \geq N$, then $\Gamma$ may be chosen to be
\[ \Gamma = \frac{400 \nu^{12} 2^{\nu}}{r^2 \omega^{2\nu}}. \]
\end{theorem}
The conditions of the theorem essentially say that $P$ looks similar to the most symmetric case possible.  Suppose $k_1 = k_2 = \ldots = k_{\nu}$ and every entry of $b$ is equal to $k^{\nu-1}/2$.  In this case we can take $k = k_j$ for all $j=1,\ldots, \nu$, and $\omega = 1$.  Furthermore, by the symmetry of the problem, we get $\zeta_j = 1/2$ for $j=1,\ldots, n$, and $\zeta_j - \zeta_j^2 = \frac{1}{4}$ for all values of $j$.  The value of $\omega$ measures how far from a hypercube the shape of the polytope's arrays are. The value of $r$  measures how far from $k^{\nu-1}/2$ the entries of $b$ are - as the entries of $b$ approach the extremal permissible values of $k^{\nu-1}$ and $0$ where the counting problem is trivial, the value of $r$ goes to zero.  It is easily seen $r$ can never be larger than $1/4$ as by hypothesis, if $z\in P\cap[0,1]^n$ then $0 \leq \zeta_j \leq 1$ for all $j=1,\ldots,n$.  The inequality $\omega k \geq 2$ is trivial, as if any $k_j = 1$, then the entries of $P$ are uniquely determined.  The non-trivial relationship between $k$, $r$, $\omega$ and $\nu$ is a consequence of how the proof of the theorem is constructed, and is likely not optimal.  However, under the heuristic described in Section ~\ref{ss:maxEntropyHeuristic}, we also would not expect Theorem ~\ref{BinaryPoints} to hold if $r$ is small enough compared to $k$.

To prove Theorem ~\ref{BinaryPoints}, we show that the number of binary integer points in $P$ can be expressed using 
\[ \int_{\Pi} F(t) dt, \]
where $\Pi \subset \mathcal{L}$ for $\mathcal{L}$ as in ~\eqref{eq:DefinitionOfL} is a cube centered at the origin whose sides have length $2\pi$, and $F(t)$ is a function that will be defined later. $F(t)$ will be similar to but different than the one described after Theorem ~\ref{IntegerPoints}, and the notation for each one will be restricted to the sections containing the proof of each theorem.  We then split $\mathcal{L}$ into three regions $X_1$, $X_2$, and $X_3$. We show that
\[ \int_{(X_2\cup X_3)\cap \Pi}  |F(t)| dt\ \ \ \text{and}\ \ \  \int_{X_2\cup X_3} e^{-q(t)}dt \ll \int_{\mathcal{L}} e^{-q(t)} dt, \]
where $q(t)$ is the quadratic form constructed in Theorem ~\ref{BinaryPoints}, and that
\[ \int_{X_1} F(t) dt \approx \int_{X_1} e^{-q(t)}dt \approx \int_{\mathcal{L}} e^{-q(t)} dt.\]
To facilitate these calculations we will require several results about the probability distribution whose density is proportional to $e^{-q(t)}$ on $\mathcal{L}$.  Sections ~\ref{s:Eigenspace}, ~\ref{s:Correlations} and ~\ref{s:ThirdDegreeTerm} will contain these, and the proof of Theorem ~\ref{BinaryPoints} will take place in Section ~\ref{s:BinaryPoints}.
\subsection{Polynomial Time Calculations}
In Theorems ~\ref{IntegerPoints} and ~\ref{BinaryPoints}, to calculate the given estimate one must calculate the determinant of a known quadratic form, which can be done in time polynomial in $n$, and find the extremal value of a strictly concave function.  This can be calculated to within error $\epsilon$ in time polynomial in $n$ and $\ln(1/\epsilon)$, see \cite{interiorPointOptimization}.  Combined, this says that both theorems give polynomial time algorithms for estimating the number of integer points or binary integer points, of transportation polytopes.

\subsection{Future Work}\label{ss:FutureWork}
In Theorem ~\ref{IntegerPoints}, the value $R$ cannot be too large or the hypothesis of the theorem is not satisfied.  This is likely an artifact of the proof technique and not a hard requirement.  In \cite{2dcase}, Barvinok and Hartigan count the number of integer points  in $2$-way transportation polytopes as long as $R/r$ is held constant, and $R$ is bounded by any arbitrary polynomial in $k$. For very large values of $r$ and $R$, the polytope itself is large enough that the volume and number of integer points approximate each other quite well.  The authors use scaling of the polytope to show that any value of $R$ is admissible, and to estimate the volume of $2$-way transportation polytopes as well with no upper bound on the value of $R$.  

\

The extra flexibility comes from being able to show that 
\[ \int_{X_2\cup X_3} F(t) dt \ll \int_{\mathcal{L}} e^{-q(t)} dt \]
for a much larger set $X_2\cup X_3$ than we are able to construct for $\nu \geq 3$. It is an open question if for $\nu \geq 3$ there is no upper bound on how large $R$ can be for the number of integer points and volume calculations.  It is also an open question if there exists a formula for the volume of $3$-way transportation polytopes even in the case when $R$ is held constant as $k$ grows.

\section{Eigenspace of the Quadratic Form}\label{s:Eigenspace}
For the entirety of this section and the next several, we will let $q(t):\mathbb{R}^{k_1+\ldots+ k_{\nu}}\to \mathbb{R}$ be the quadratic form 
\begin{equation}\label{eq:Definitionofq(t)} q(t) = \frac{1}{2}\sum_{m_1,\ldots,m_{\nu}=1}^{k_1,\ldots,k_{\nu}} \alpha_{m_1,\ldots,m_{\nu}}(t_{1 m_1}+\ldots t_{\nu m_{\nu}})^2, \end{equation}
where $\alpha_{m_1,\ldots,m_{\nu}}$ are arbitrary positive constants, and we assume that each $k_i$ is at least $2$.
Note that the quadratic forms in Theorems ~\ref{IntegerPoints} and ~\ref{BinaryPoints} are of this form with $\alpha_{m_1,\ldots,m_{\nu}} = \zeta_{m_1,\ldots,m_{\nu}} +\zeta_{m_1,\ldots,m_{\nu}}^2$ in the first case and $\alpha_{m_1,\ldots,m_{\nu}} = \zeta_{m_1,\ldots,m_{\nu}} - \zeta_{m_1,\ldots,m_{\nu}}^2$ in the second.  We will let $B$ be the unique positive semidefinite matrix such that 
\[q(t) = \frac{1}{2} \left<t,Bt\right>.\]
Note that for $D$ as in Theorem ~\ref{IntegerPoints} or ~\ref{BinaryPoints}, we have $B = D^t D$.
Suppose $X$ is a random variable whose density is proportional to $e^{-q(t)}$ restricted to $\mathcal{L}$, where $\mathcal{L}$ is as defined in ~\eqref{eq:DefinitionOfL}.  The objective of the next several sections is to calculate correlations of random variables of the form $\left<X,e_i \right>$ for any standard basis vector $e_i \in \mathcal{L}$.  To do so, we will carefully bound the  eigenvalues of $q(t)$ and estimate the eigenvectors of $q(t)$.

\

The application of these results will be applied in the proofs of the main theorems in two ways. It will allow us to show that the integral of $e^{-q(t)}$ outside of a neighborhood of the origin is negligible. It will also allow us to place bounds on $\mathbf{E}\left(e^{if(t)} \right)$ for a certain cubic polynomial $f(t)$ when $t$ is drawn from the distribution given by $X$.  In the proof of both theorems we will show the number of integer or binary integer points is equal to the integral of a function $F(t)$ (different for each theorem), which we will approximate near the origin via Taylor polynomial approximations as $F(t) \approx e^{-q(t)+if(t)}$.  The results of these next several sections will allow us to then estimate the integral of $F$ in a neighborhood of the origin.

\

We introduce some notation and concepts.  If $C$ is a symmetric matrix, we write $\lambda_i(C)$ to be the $i$th largest eigenvalue of $C$.  Throughout the entire section we assume that there are values $R>r>0$ such that 
\[r \leq \alpha_{m_1,\ldots,m_{\nu}} \leq R\ \ \ \text{for all}\ \ \ m_1,\ldots,m_{\nu}.\]  
We also assume there exists $\omega$ and $k$ such that 
\[1 \leq\omega k \leq k_1,\ldots,k_{\nu} \leq k.\]  
For notational convenience we will define
\[ k'_j = \prod_{\substack{ i = 1,\ldots,\nu \\ i\neq j}} k_i. \]
We let $Q:\mathbb{R}^{k_1+\ldots+k_{\nu}} \to \mathbb{R}^{k_1+\ldots + k_{\nu}}$ be the orthogonal projection onto $\mathcal{L}$.
 The main result of this section is the following:

 \begin{theorem}\label{thm:RestrictedEigenspace} For $\nu \geq 2$ there exists a set of eigenvectors and eigenvalues of $QBQ$ as follows: there are $\nu-1$ eigenvectors with eigenvalue $0$ lying in the kernel of $Q$,
$\nu-1$ unit eigenvectors with eigenvalues that lie between
\[ r\frac{ \omega^{\nu-1}}{\nu(\nu-1)} k^{\nu-2}\ \ \ \text{and}\ \ \ R\omega^{-1} k^{\nu-2} \] 
such that the square of the distance of each eigenvector to $\ker(B)$ is smaller than
\[\frac{R}{r} \omega^{-\nu} k^{-1},\]
one eigenvalue which lies between
\[ \frac{r}{2}\omega^{\nu-1} \nu k^{\nu-1}\ \ \ \text{and}\ \ \ R\nu k^{\nu-1}, \]
and the remaining eigenvalues all lie between
\[ r\omega^{\nu-1} k^{\nu-1}\ \ \ \text{and}\ \ \ R k^{\nu-1}. \] \end{theorem}
This theorem describes the eigenvalues and eigenvectors of the quadratic form $q|_{\mathcal{L}}$.  The outline of the proof is as follows: we first calculate all the eigenvectors and eigenvalues of $B$, and see the eigenvalues are all $\Theta(k^{\nu-1})$.  Most of these eigenvectors will lie in $\mathcal{L}$ and hence be eigenvectors of $q|_{\mathcal{L}}$ as well.  The remaining few eigenvectors will be nearly orthogonal to $\mathcal{L}$, which we will use to show that the remaining eigenvalues are $\Theta(k^{\nu-2})$.

We require the use of two well known lemmas on comparing eigenvalues of symmetric matrices.
\begin{lemma} \label{lem:Horn}Let $C$ and $D$ be symmetric positive semidefinite $m\times m$ matrices such that $C-D$ is positive semidefinite. Then
\[ \lambda_{i}(C) \geq \lambda_{i}(D) \text{ for all } i=1,\ldots,m.\]
\end{lemma}
\begin{proof}This is Corollary 7.7.4 of \cite{Horn}.\end{proof}

\begin{lemma}\label{lem:WeylInequalities}\emph{(The Weyl Inequalities)}: Let $C$ and $D$ be $m\times m$ real symmetric matrices.  Then
\[ \lambda_{i+j-1}(D+C) \leq \lambda_i(C) + \lambda_j(D) \]
as long as $1\leq i,j \leq m$ such that $i+j-1 \leq m$.\end{lemma}
\begin{proof}This inequality is shown in Section 1.3.3 of \cite{terrytao}.\end{proof}

\begin{lemma} \label{lem:EigenvaluesOfq(t)}  The matrix $B$ has a basis of orthogonal eigenvectors such that $\nu-1$ lie in the kernel of $B$, one has eigenvalue between $r \omega^{\nu-1} \nu k^{\nu-1}$ and $R \nu k^{\nu-1}$, and the remaining eigenvalues lie between $r \omega^{\nu-1}k^{\nu-1}$ and $R k^{\nu-1}$. \end{lemma}
\begin{proof}If $q(t) = \frac{1}{2}\left<t, Bt \right>$ then 
\[ \nabla q(t) = Bt. \]
Calculating the gradient in ~\eqref{eq:Definitionofq(t)} gives us
 \begin{equation}\label{eq:GeneralGradientOfq} \frac{\partial q}{\partial \tau_{j m_j}} = \sum_{m_1,\ldots, \hat{m_j},\ldots, m_{\nu}} \alpha_{m_1,\ldots,m_{\nu}}(\tau_{1 m_1} + \ldots + \tau_{\nu m_{\nu}}). \end{equation}
First we consider the case when $\alpha_{m_1,\ldots,m_{\nu}} = 1$ for all $m_1,\ldots,m_{\nu}$.  Then for all $1 \leq j \leq \nu$ and $1 \leq m_j \leq k_j$,
\begin{equation}\label{eq:Nablaqexpanded}\frac{\partial q}{\partial \tau_{j m_j}} = k'_j \tau_{j m_j} + \sum_{m_1,\ldots,\hat{m}_j,\ldots,m_{\nu}=1}^{k_1,\ldots,\hat{k}_j,\ldots,k_{\nu}} \left(\tau_{1 m_1}+\ldots+\hat{\tau}_{j m_j}+\ldots+\tau_{\nu m_{\nu}} \right). \end{equation}
From this it is immediately clear by substituting in the relevant vectors that for any $1\leq j \leq \nu$, any non-zero vector contained in the subspace
\begin{equation}\label{eq:DefinitionOfWj} \mathcal{W}_j = \left\{(\underbrace{0,\ldots,0}_{k_1+\ldots + k_{j-1}},\tau_{j1},\ldots,\tau_{jk_{j}},\underbrace{0,\ldots,0}_{k_{j+1}+\ldots+k_{\nu}})\ :\ \sum_{i=1}^{k_j} \tau_{j i} = 0 \right\}\end{equation}
is an eigenvector of $B$ with eigenvalue $k'_j$. As $B$ has a basis of orthogonal eigenvectors, we can complete the description of its eigenspace by considering vectors orthogonal to each $\mathcal{W}_j$, which are of the form \[(\underbrace{\sigma_1,\ldots,\sigma_1}_{k_1},\underbrace{\sigma_2,\ldots,\sigma_2}_{k_2},\ldots,\underbrace{\sigma_{\nu},\ldots,\sigma_{\nu}}_{k_{\nu}}).\] 
If $\sigma_1+\ldots+\sigma_{\nu} = 0$ then this vector lies in the kernel of $B$, so is an eigenvector with eigenvalue $0$.  By dimension counting there is one remaining eigenvector of $B$, which has $\sigma_j = k'_j$ for all $j$ and has an eigenvalue of
\[ \sum_{j=1}^{\nu} k'_j. \]

If instead we have $r < \alpha_{m_1\ldots m_{\nu}} < R$, the set of vectors of the form 
\[(\underbrace{\sigma_1,\ldots,\sigma_1}_{k_1},\underbrace{\sigma_2,\ldots,\sigma_2}_{k_2},\ldots,\underbrace{\sigma_{\nu},\ldots ,\sigma_{\nu}}_{k_{\nu}})\ \ \text{ such that }\ \ \sigma_1+\ldots+\sigma_{\nu} = 0\]  
still form the kernel of $B$, and the remaining eigenvectors will be orthogonal to this space.  Applying Lemma ~\ref{lem:Horn} and comparing the eigenvalues of $q(t)$ with the quadratic forms 
\[ \tilde{q}_Z(t) = \frac{Z}{2} \sum_{m_1,\ldots, m_{\nu}=1}^{k_1,\ldots, k_{\nu}} \left(\tau_{1 m_1} + \ldots + \tau_{\nu m_{\nu}} \right)^2 \]
for $Z = R $ and $Z=r$ completes the proof.  
\end{proof}

At this point we are ready to prove Theorem ~\ref{thm:RestrictedEigenspace}.
\begin{proof} We will first prove the result when $\alpha_{m_1,\ldots, m_{\nu}} = 1$ for all $m_1,\ldots,m_{\nu}$. Then for $\mathcal{W}_j$ as defined in Equation ~\eqref{eq:DefinitionOfWj} we immediately get that $\mathcal{W}_j \cap \mathcal{L}$
are eigenspaces of $QBQ$ with eigenvalues $k'_j$.  Furthermore, $\mathcal{W}_1 \cap \mathcal{L} = \mathcal{W}_1$ has dimension $k_1$, and for $j\geq 2$, the subspace $\mathcal{W}_j\cap \mathcal{L}$ has codimension $1$ in $\mathcal{W}_j$ so is a subspace of $\mathcal{L}$ of dimension $k_j-1$.  By dimension counting we are left with $\nu$ linearly independent eigenvectors of $QBQ$ in $\mathcal{L}$ that are unaccounted for.  There must exist a set of eigenvectors of the form 
\begin{equation} \label{eq:remainingxform} s = (\underbrace{\sigma_1,\ldots ,\sigma_1}_{k_1},\underbrace{\sigma_2,\ldots ,\sigma_2}_{k_2-1},0,\underbrace{\sigma_3,\ldots,\sigma_3}_{k_3-1},0,\ldots ,\underbrace{\sigma_{\nu},\ldots ,\sigma_{\nu}}_{k_{\nu}-1},0)\end{equation}
as they are orthogonal to the eigenspaces of $QBQ$ that we have calculated so far. Let $\mathcal{V}$ be the subspace of all vectors of the form defined in Equation ~\eqref{eq:remainingxform}.  We decompose $s$ into a linear combination of the $2\nu - 1$ remaining eigenvectors of $B$ orthogonal to $\mathcal{W}_j\cap \mathcal{L}$ for all $j$.  These are the kernel vectors 
\begin{equation}\label{eq:KernelVectorForm}(\underbrace{\sigma_1,\ldots ,\sigma_1}_{k_1},\underbrace{\sigma_2,\ldots , \sigma_2}_{k_2}, \ldots ,\underbrace{\sigma_{\nu},\ldots ,\sigma_{\nu}}_{k_{\nu}})\ \ \text{ with }\ \ \sum_{j=1}^{\nu} \sigma_j = 0,\end{equation}
the vector 
\begin{equation} \label{eq:DefinitionOfV0}v_0 = (\underbrace{k'_1,\ldots ,k'_1}_{k_1},\underbrace{k'_2,\ldots ,k'_2}_{k_2},\ldots ,\underbrace{k'_{\nu},\ldots ,k'_{\nu}}_{k_{\nu}}) \end{equation}
and one vector from each $\mathcal{W}_j$ for $j\neq 1$ of the form
\begin{equation} \label{eq:DefinitionOfVj} v_j = (\underbrace{0,0\ldots,0}_{k_1+\ldots+ k_{j-1}}, \underbrace{1,1,\ldots ,1}_{k_j - 1},1-k_j,\underbrace{0,0,\ldots,0}_{k_{j+1}+\ldots+ k_{\nu}}). \end{equation}
The projection of $s$ onto the span of $v_j$ by definition is
\[  \frac{ \left< v_j, s \right>}{\left<v_j, v_j \right>} v_j.\]
For $s\in \im(Q)$, we have $\left<v_j, s \right> = \left<Qv_j, s \right>$.  If $P_j$ is the projection onto $v_j$, then
\[ QBQs = QBQ\left(P_0 s + \sum_{j=2}^{\nu} P_j s \right)\]
can be rewritten as
\begin{equation} \label{eq:QBQsDecomposed} QBQs=\frac{\left< Qv_0, s \right>}{\left<v_0,v_0 \right>}  QBQ v_0 + \sum_{j\geq 2} \frac{\left<Qv_j, s \right>}{\left<v_j, v_j \right>} QBQv_j.   \end{equation}

Let 
\begin{equation} \label{eq:vTou1}u_0 = Q v_0 = (\underbrace{k'_1,..,k'_1}_{k_1},\underbrace{k'_2,\ldots,k'_2}_{k_2-1},0,\underbrace{k'_3,\ldots,k'_3}_{k_3-1},0,\ldots,\underbrace{k'_{\nu},\ldots,k'_{\nu}}_{k_{\nu}-1},0), \end{equation}
and for $j > 1$,
\begin{equation} \label{eq:vTou2} u_j = Q v_j = (\underbrace{0,\ldots,0}_{k_1+\ldots+k_{j-1}},\underbrace{1,\ldots,1}_{k_j-1},0,\underbrace{0,\ldots,0}_{k_{j+1}+\ldots+k_{\nu}}). \end{equation}
By the eigenvalues of the $v_j$s calculated in Lemma ~\ref{lem:EigenvaluesOfq(t)}, plugging Equations ~\eqref{eq:vTou1} and ~\eqref{eq:vTou2} into Equation ~\eqref{eq:QBQsDecomposed}, we get
\[ QBQ|_{\mathcal{V}} = \frac{ \sum_{j =1}^{\nu} k'_j}{\left<v_0, v_0 \right>} u_0 u_0^t + \sum_{j\geq 2} \frac{k'_j}{\left<v_j, v_j \right>} u_j u_j^t. \]
We can then write $QBQ|_{\mathcal{V}} = C+D$, where
\[C = \frac{ \sum_{j =1}^{\nu} k'_j}{\left<v_0, v_0 \right>} u_0 u_0^t,\]
is a rank one symmetric matrix with nonzero eigenvector $u_0$ and eigenvalue
\[ \lambda_C = \frac{  \left<u_0, u_0 \right>}{\left<v_0, v_0\right>}\sum_{j =1}^{\nu} k'_j,  \]
and 
\[ D = \sum_{j\geq 2} \frac{k'_j}{\left<v_j, v_j \right>} u_j u_j^t\]
 is a rank $\nu-1$ symmetric matrix that has nonzero eigenvectors $u_j$ for $j\geq 2$ of eigenvalue 
\[ \lambda_j = \frac{ \left< u_j, u_j \right> k'_j}{\left< v_j, v_j \right>}. \]

By Equations ~\eqref{eq:DefinitionOfV0} and ~\eqref{eq:vTou1},
\begin{equation} \label{eq:BoundsOnEigenvalueOfC} \left(1-\frac{1}{k} \right) \nu \omega^{\nu-1} k^{\nu-1} \leq \lambda_C \leq  \left(1-\frac{\omega}{k} \right) \nu k^{\nu-1},
 \end{equation}  
and by Equations ~\eqref{eq:DefinitionOfVj} and ~\eqref{eq:vTou2}, for all $j \geq 2$ we have
\begin{equation} \label{eq: BoundsOnEigenvaluesOfD} \omega^{\nu-1} k^{\nu-2} \leq \lambda_{j} \leq \omega^{-1} k^{\nu-2}. \end{equation}

By Lemma ~\ref{lem:WeylInequalities}, we get by plugging in $i=2$ and $j = 1$ that
\begin{equation} \label{eq:UpperBoundLambda2}\lambda_2\left(QBQ|_\mathcal{V}\right) \leq \lambda_{2}\left(C|_\mathcal{V}\right) + \lambda_1\left(D|_\mathcal{V}\right) \leq \omega^{-1} k^{\nu-2}. \end{equation}
Furthermore,
\[ \lambda_1\left(QBQ|_\mathcal{V}\right) \geq \lambda_1\left(C|_\mathcal{V}\right) \geq \left(1-\frac{1}{k} \right) \nu \omega^{\nu-1} k^{\nu-1}. \]
Also, the largest eigenvalue of $QBQ|_\mathcal{V}$ must be smaller than the largest eigenvalue of $B$, which is $\nu k^{\nu-1}$.  Taking the largest eigenvalue of $QBQ|_\mathcal{V}$ and the eigenvalues we have calculated previously, we find that all but $\nu-1$ eigenvalues of $QBQ$ on $\mathcal{L}$ lie between
\[  \omega^{\nu-1} k^{\nu-1}\ \ \ \text{and}\ \ \ \nu k^{\nu-1}. \]  
Furthermore, by ~\eqref{eq:UpperBoundLambda2}, to complete the proof on the magnitude of the eigenvalues it suffices to show that 
\[ \lambda_{\nu}\left(QBQ|_\mathcal{V}\right) \geq \frac{\omega^{\nu-1}}{\nu(\nu-1)}k^{\nu-2}. \] 
If $t\in \mathcal{L}$ is an eigenvector of $QBQ|_\mathcal{V}$ with eigenvalue $\lambda$, then $q(t) = \frac{1}{2}\lambda ||t||^2.$  Decomposing $t = w_1 + w_2$ with $w_1\in \ker(B)$ and $w_2 \perp \ker(B)$, we also get $ q(t) = q(w_2) \geq \frac{1}{2}\omega^{\nu-1} k^{\nu-1} ||w_2||^2$ by Lemma ~\ref{lem:EigenvaluesOfq(t)}.  Combining these gives
\begin{equation}\label{eq:lambdaBoundWithw2} \lambda \geq \omega^{\nu-1} k^{\nu-1} \frac{||w_2||^2}{||t||^2}. \end{equation}
 
Assume $t$ is of the form given in Equation ~\eqref{eq:remainingxform}.  Let $T$ be the orthogonal projection onto the kernel of $B$, so $w_1 = Tt$.  Then 
\begin{equation}\label{eq:w2BoundEqualsOneMinus} \frac{||w_2||^2}{||t||^2} = 1-\frac{\left<t,u\right>^2}{||t||^2}. \end{equation}
Furthermore,
\[ Tt = \left<t,u \right> u\ \ \ \text{for}\ \ \ u = \frac{1}{||w_1||}w_1, \]
and
\[ \frac{ ||Tt ||^2}{||t||^2} = \frac{\left<t,u \right>^2}{||t||^2}. \]
We also use the fact that
\[ ||Qu||^2 \geq \frac{\left<t,u \right>^2}{||t||^2} \]
as the projection of $u$ onto $\mathcal{L}$ is at least as large as the projection of $u$ onto the span of $t$.  Combining these along with ~\eqref{eq:lambdaBoundWithw2}, we get that
\begin{equation}\label{eq:FinalLambdaBound} \lambda_{\nu}\left(QBQ|_{\mathcal{V}} \right) \geq \omega^{\nu-1}k^{\nu-1} \left(1- ||Qu||^2 \right). \end{equation}
We consider the problem
\[ \text{minimize}\ \ \  ||(\Id-Q)u||^2 \ \ \ \text{given}\ \ \ u\in \ker(B),\ ||u|| = 1. \] 
Recall that every $u\in \ker(B)$ is in the form given in Equation ~\eqref{eq:KernelVectorForm}, so the problem reduces to finding a lower bound for
\[||(\Id-Q)u||^2 = \sum_{j=2}^{\nu} \sigma_j^2 \] 
under the conditions that 
\[\sum_{j=1}^{\nu} k_j \sigma_j^2= 1\ \ \text{ and }\ \ \sum_{j=1}^{\nu} \sigma_j = 0.\]  
Substituting $-\sigma_1=\sigma_2+\ldots+\sigma_{\nu}$, we reduce the problem to
\[ \text{minimize}\ \ \  \sum_{j=2}^{\nu} \sigma_j^2 \ \ \ \text{given}\ \ \  k_1 \left( \sum_{j=2}^{\nu} \sigma_j \right)^2 + \sum_{j=2}^{\nu} k_j \sigma_j^2= 1. \]
If the minimum is achieved at $(\sigma_2,\ldots,\sigma_{\nu}) = (\beta_2,\ldots,\beta_{\nu})$, then every $\beta_j$ must have the same sign. If not, then 
\[ k_1\left(\sum_{j=2}^{\nu} |\beta_j| \right)^2+\sum_{j=2}^{\nu} k_j \beta_j^2 \geq k_1 \left( \sum_{j=2}^{\nu} \beta_j \right)^2+\sum_{j=2}^{\nu} k_j \beta_j^2 = 1,\]
 and therefore we can take the vector $(|\beta_2|,\ldots,|\beta_{\nu}|)$ and scale it down to find a smaller minimum satisfying the constraints.  By taking negatives if necessary, assume $\beta_j > 0$ for all $j\geq 2$.  If for some $i$ we have $\beta_i \geq \beta_j$ for all $j\geq 2$, then
 \[ \left((\nu-1)^2 k_1 + \sum_{j=2}^{\nu} k_j \right) \beta_i^2\geq k_1\left(\sum_{j=2}^{\nu} \beta_j \right)^2+\sum_{j=2}^{\nu} k_j \beta_j^2 = 1,\]
 so 
 \[\beta_i^2 \geq \frac{1}{(\nu-1)^2 k_1 + \sum_{j=2}^{\nu} k_j} \]
and hence
\[ \sum_{j=2}^{\nu} \beta_j^2 \geq \frac{1}{(\nu-1)^2 k_1 + \sum_{j=2}^{\nu} k_j}.\] Therefore 
 \[||(\Id-Q)u||^2 \geq \frac{1}{\nu(\nu-1)k} ||u||^2,\]
 which combined with Equation ~\eqref{eq:FinalLambdaBound} completes the proof that on $\mathcal{V}$, 
 \[\lambda_{\nu}(QBQ|_{\mathcal{V}}) \geq \frac{\omega^{\nu-1}}{\nu(\nu-1) k^{\nu-2}}. \]
Combining this with ~\eqref{eq:UpperBoundLambda2} completes the proof of the theorem in the case that $\alpha_{m_1,\ldots, m_{\nu}}= 1$ for all $m_1,\ldots,m_{\nu}$.

\

By Lemma ~\ref{lem:Horn}, if $r \leq \alpha_{m_1,\ldots,m_{\nu}} \leq R$, the eigenvalues of $q|_{\mathcal{L}}$ are bounded from above by the previously calculated eigenvalues multiplied by $R$, and bounded from below by the previously calculated eigenvalues multiplied by $r$.  Furthermore if $t\in \mathcal{L}$ is a unit eigenvector whose eigenvalue is smaller than $R\omega^{-1} k^{\nu-2}$, then writing
\[ t = \alpha_t t_1 + \beta_t t_2 \]
with $t_1 \in \ker(B)$, $t_2 \in \im(B)$, we get that
\[ \frac{1}{2}R\omega^{-1} k^{\nu-2} \geq q(t) \geq \frac{1}{2}\beta_t^2 r \omega^{\nu-1}k^{\nu-1}, \]
and hence 
\[ \beta_t^2 \leq \frac{R}{r} \omega^{-\nu} k^{-1}, \]
which completes the proof.

\end{proof}
We immediately get the following corollary:
\begin{corollary}\label{cor:BoundOnGaussianIntegral} 
\[\int_{\mathcal{L}} e^{-q(t)} dt \geq \exp\left(-\frac{1}{4}\nu^2 k \ln(k)-\frac{1}{2}\nu k \ln(R)\right).\]
\end{corollary}

\begin{proof}
First, we observe by ~\eqref{eq:IntegralOfGaussian} and Theorem ~\ref{thm:RestrictedEigenspace} that
\[ \int_{\mathcal{L}} e^{-q(t)} dt \geq (2\pi)^{(k_1+\ldots+k_{\nu}-\nu+1)/2} \sqrt{\left(\frac{\omega}{Rk^{\nu-2}}\right)^{(\nu-1)} \frac{1}{R\nu k^{\nu-1}} \left(\frac{1}{Rk^{\nu-1}}\right)^{k_1+\ldots+k_{\nu}-2\nu+1} }.\]
Observing that $\omega k \geq 1$, we can simplify this to
\[ \int_{\mathcal{L}} e^{-q(t)}dt \geq (2\pi)^{(k_1+\ldots+k_{\nu}-\nu+1)/2}\frac{1}{\sqrt{\nu}} \left(Rk^{\nu-1}\right)^{-(k_1+\ldots+k_{\nu}-\nu+1)/2},\]
which we can further simplify by using the fact that $k\geq 2$ to the claim of the corollary.

\end{proof}

\section{Variance of the Gaussian}\label{s:Variance}
In this section we continue to use all notation introduced in the beginning of Section ~\ref{s:Eigenspace}.  In particular the quadratic form $q(t)$, the subspace $\mathcal{L}$, and the constants $r$, $R$, $\omega$ and $k$, as well as the notation for $k'_j$.  We consider the probability density on $L$ proportional to $e^{-q(t)}$, and show the total measure outside of a small box around the origin in $\mathcal{L}$ is negligible.  The main result is the following:
\begin{lemma} \label{lemma:MainVariance}Let
\[ X_{\delta} = \left\{ t\in \mathcal{L}\ :\ ||t||_{\infty} \geq \delta \right\}. \]
Then
\[ \int_{X_{\delta}} e^{-q(t)} dt \leq \nu k  \exp \left(- \delta^2k^{\nu - 1}/\Gamma \right) \int_{\mathcal{L}} e^{-q(t)} dt,\]
where $dt$ is the Lebesgue measure on $\mathcal{L}$, and 
\[ \Gamma = \frac{2\nu^4 R}{\omega^{6\nu-3} r^2}.\]  \end{lemma}
To prove Lemma ~\ref{lemma:MainVariance}, we consider random variables of the form $\left<t,v \right>$ for a fixed vector $v$ when $t$ is drawn from the distribution with density proportional to $e^{-q(t)}$ restricted to $\mathcal{L}$.  In general, if $\psi(t)$ is a positive definite quadratic form on a vector space $\mathcal{V}$ of dimension $d$ with unit eigenvectors $v_1,\ldots,v_d$ and eigenvalues $\lambda_1,\ldots,\lambda_d$, and $t$ is drawn randomly from the distribution whose density is proportional to $e^{-\psi(t)}$ on $\mathcal{V}$, and $u \in \mathcal{V}$ is fixed, then $\left<u,t \right>$ is a normal random variable, and
\begin{equation}\label{eq:VarianceInDirection}\mathbf{Var}\left(\left<u,t\right> \right) = \sum_{j=1}^{d} \frac{1}{\lambda_i}\left<u, v_i \right>^2. \end{equation}
Before we begin the proof of the main result we require a technical lemma.
\begin{lemma}\label{lem:InfinityNormOfProjection} Let $e_j$ be a standard basis vector of $\mathbb{R}^{k_1+\ldots+k_{\nu}}$, and let $T$ be the orthogonal projection onto the kernel of $q(t)$.  Then there exist constants $\gamma(\omega,\nu) > 0$, $\Gamma(\omega,\nu) > 0$ such that 
\[ \frac{\gamma}{k}\ \leq\ ||T e_j||_{\infty}\ \leq\ \frac{\Gamma}{k}. \]
The constants may be chosen to be
\[ \gamma = \frac{  \omega^{2\nu-1}}{2} ,\ \ \ \text{and}\ \ \  \Gamma = \frac{1} {\omega^{2\nu-1}}. \]
\end{lemma}
\begin{proof}
For notational simplicity we assume that $e_j$ corresponds to one of the first $k_1$ entries.  An orthogonal basis of the kernel of $q(t)$ can be written as follows:
\[ u_{\nu-1} = (\underbrace{0,\ldots,0}_{k_1+\ldots+k_{\nu-2}}, \underbrace{-k'_{\nu},\ldots,-k'_{\nu}}_{k_{\nu-1}},\underbrace{k'_{\nu},\ldots,k'_{\nu}}_{k_{\nu}}), \]
\[ u_{\nu-2} = (\underbrace{0,\ldots,0}_{k_1+\ldots+k_{\nu-3}}, \underbrace{-(k'_{\nu-1} + k'_{\nu}),\ldots,-(k'_{\nu-1} + k'_{\nu})}_{k_{\nu-2}}, \underbrace{k'_{\nu-1},\ldots,k'_{\nu-1}}_{k_{\nu-1}},\underbrace{k'_{\nu},\ldots,k'_{\nu}}_{k_{\nu}}), \]
and for any $2\leq i \leq \nu$, we have $u_{i-1}$ given by
\[(\underbrace{0,\ldots,0}_{k_1+\ldots+ k_{i-2}},\underbrace{-(k'_i+\ldots+k'_{\nu}),\ldots,-(k'_i+\ldots+k'_{\nu})}_{k_{i-1}}, \underbrace{k'_i,\ldots,k'_i}_{k_i},\ldots,\underbrace{k'_{\nu},\ldots,k'_{\nu}}_{k_{\nu}}), \]
culminating with
\[ u_1 = (\underbrace{-(k'_2+\ldots+k'_{\nu}), \ldots, -(k'_2+\ldots+k'_{\nu})}_{k_1},\underbrace{k'_2,\ldots,k'_2}_{k_2},\underbrace{k'_3,\ldots,k'_3}_{k_3},\ldots,\underbrace{k'_{\nu},\ldots,k'_{\nu}}_{k_{\nu}}). \]
It is easy to see by construction that each $u_i$ lies in the kernel of $q(t)$, and by dimension counting they therefore form a basis.  To see that they form an orthogonal set, for any $i<l$ we have
\[ \left<u_i, u_l \right> = -k_l k'_l \left( \sum_{p=l+1}^{\nu} k'_p \right) + \sum_{p=l+1}^{\nu} k_p (k'_p)^2. \]
As $k_p k'_p = k_1 k_2\ldots k_{\nu}$ for any $p$, we can re-write this as
\[ -(k_1 k_2\ldots k_{\nu}) \left( \sum_{p=l+1}^{\nu} k'_p \right) + (k_1 k_2\ldots k_{\nu}) \sum_{p=l+1}^{\nu} k'_p  = 0. \]
Then $e_j$ is orthogonal to $u_i$ for all $i>1$, and therefore the projection of $e_j$ onto the kernel of $q(t)$ is simply
\begin{equation}\label{eq:ProjectionOntKernelOfQ}\frac{ \left<e_j, u_1 \right>}{\left<u_1, u_1 \right>} u_1 = \frac{ -(k'_2+\ldots+k'_{\nu})}{ k_1 (k'_2+\ldots+k'_{\nu})^2 + k_2 k'^2_2+\ldots+k_{\nu} k'^2_{\nu}} u_1. \end{equation}
Using the inequality $\omega k \leq k_1,\ldots, k_{\nu} \leq k$, we get
\[  \frac{(\nu-1) \omega^{\nu}}{(\nu-1)^2 + \nu-2} k^{-\nu} \leq \left| \frac{ -(k'_2+\ldots+k'_{\nu})}{ k_1 (k'_2+\ldots+k'_{\nu})^2 + k_2 k'^2_2+\ldots+k_{\nu} k'^2_{\nu}} \right|,\ \ \text{ and}  \]
\[\left| \frac{ -(k'_2+\ldots+k'_{\nu})}{ k_1 (k'_2+\ldots+k'_{\nu})^2 + k_2 k'^2_2+\ldots+k_{\nu} k'^2_{\nu}} \right| \leq \frac{ (\nu-1)}{ \omega^{2\nu-1} \left((\nu-1)^2 + \nu-2\right)} k^{-\nu}. \]
Combining these with
\[ (\nu-1) \omega^{\nu-1} k^{\nu-1} \leq ||u_1||_{\infty} \leq (\nu-1) k^{\nu-1} \]
in ~\eqref{eq:ProjectionOntKernelOfQ} and simplifying bounds completes the proof.
\end{proof}
 \begin{lemma} \label{lem:Variance1}Suppose $t$ is drawn from the distribution with density proportional to $e^{-q(t)}$ restricted to $\mathcal{L}$.  Then if $e_j$ is a standard basis vector contained in $\mathcal{L}$, there exists a constant $\Gamma = \Gamma(\omega,\nu,r,R) > 0$ such that 
\[ \mathbf{Var}\left(\left<t,e_j \right> \right) \leq \frac{\Gamma}{k^{\nu-1}}. \]
The constant $\Gamma$ may be chosen to be
\[ \Gamma =  \frac{14 \nu^4 R}{\omega^{6\nu-3}r^2}. \] \end{lemma}
\begin{proof} 
We apply Equation ~\eqref{eq:VarianceInDirection} with $\mathcal{V}=\mathcal{L}$, letting $u=e_j$ be any standard basis vector contained in $\mathcal{L}$ and $\psi(t) = q(t)$ restricted to $\mathcal{L}$.  Let $v_1,\ldots v_{\nu-1}$ be the unit eigenvectors whose distance to $\ker(B)$ was calculated in Theorem ~\ref{thm:RestrictedEigenspace}.  Substituting in the lower bound for the eigenvalues of the remaining eigenvectors from Theorem ~\ref{thm:RestrictedEigenspace} into Equation ~\eqref{eq:VarianceInDirection}, we get the variance is bounded from above by
\begin{equation}\label{eq:AlmostDoneWithFirstVarianceCalcs} \frac{1}{r\omega^{\nu-1}k^{\nu-1}} + \frac{\nu(\nu-1)}{r\omega^{\nu-1} k^{\nu-2}}\sum_{i=1}^{\nu-1} \left<e_j, v_i \right>^2. \end{equation}
For any such $v_i$, we can decompose it as
\[ v_i = a_i + b_i, \text{ where} \]
\[ a_i \in \ker(B),\ ||a_i|| \leq 1, \text{ and} \]
\[ b_i \perp \ker(B),\ ||b_i|| \leq \sqrt{\frac{R}{r} \omega^{-\nu}} k^{-1/2}. \]
Then 
\begin{equation}\label{eq:DecomposeInnerProductOfEigenvector} \left<e_j, v_i \right>^2 = \left<e_j, a_i \right>^2 + \left<e_j, b_i \right>^2 + 2\left<e_j, a_i \right>\left< e_j, b_i \right>.\end{equation} 
 If $T$ is the orthogonal projection onto the kernel of $B$, then
\[\left| \left<e_j, a_i \right> \right|  \leq ||Te_j|| \leq \sqrt{\nu k} ||T e_j||_{\infty}. \]
Applying Lemma ~\ref{lem:InfinityNormOfProjection},
\[ \left| \left<e_j, a_i \right> \right| \leq \frac{\sqrt{\nu}}{\omega^{2\nu-1}} k^{-1/2}. \]
Also,
\[ \left| \left<e_j, b_i \right> \right| \leq   ||b_i|| \leq \sqrt{\frac{R}{\omega^{\nu}r}}  k^{-1/2}. \]
Combining these into Equation ~\eqref{eq:DecomposeInnerProductOfEigenvector} yields
\[ \left<e_j, v_i\right>^2 \leq \left(\frac{\nu}{\omega^{2\nu-1}} + \frac{R}{\omega^{\nu} r} + 2\sqrt{\frac{\nu R}{\omega^{5\nu-2}r}}\right)k^{-1}.\] 
Simplifying the bound to be
\[\left<e_j, v_i\right>^2 \leq \frac{4\nu R}{\omega^{5\nu-2}r} k^{-1}, \]
plugging this into Equation ~\eqref{eq:AlmostDoneWithFirstVarianceCalcs} and simplifying again completes the proof.\end{proof}
To calculate bounds on probabilities we use the following lemma.
\begin{lemma} \label{lemma:Variance2} Let $X$ be a normal variable with variance $\sigma^2$ and $\mathbf{E}(X) = 0$.  Then 
\[ \mathbf{Pr}\left( |X| \geq \tau\right) \leq    e^{-\tau^2/(2\sigma^2)}. \]\end{lemma}
\begin{proof}  We use the well known result that if $X$ is a standard normal variable then
\[ \mathbf{Pr}(|X| \geq \tau) \leq e^{-\tau^2/2}. \]
If $X$ has variance $\sigma^2$, then $X/\sigma$ is the standard normal variable, so
\[ \mathbf{Pr}(|X| \geq \tau) = \mathbf{Pr}(|X|/\sigma \geq  \tau/\sigma) \leq e^{-\tau^2/(2\sigma^2)}. \]
\end{proof}

Combining Lemmas ~\ref{lem:Variance1} and ~\ref{lemma:Variance2}, along with a union bound, gives Lemma ~\ref{lemma:MainVariance}.

\section{Correlations}\label{s:Correlations}
In this section we use the notation introduced at the start of Section ~\ref{s:Eigenspace}, in particular the quadratic form $q(t)$, the subspace $\mathcal{L}$, the constants $r$, $R$, $\omega$ and $k$, and the definition of $k'_j$.  If we draw
\[ t = (\tau_{11},\ldots,\tau_{1k_1},\tau_{21},\ldots,\tau_{2k_2},\ldots,\tau_{\nu 1},\ldots,\tau_{\nu k_{\nu}}) \]
from $\mathcal{L}$ with density proportional to $e^{-q(t)}$, then we can treat the individual coordinates $\tau_{ij}$ as random variables.  In this section we will calculate the correlation between pairs of coordinates.  The main result of this section is the following:
\begin{theorem}\label{thm:Correlations}Let $\mathcal{M}$ be any subspace of $\mathbb{R}^{k_1+\ldots+k_{\nu}}$ of codimension $\nu-1$ such that $\mathcal{M}\cap \ker(q) = \{ 0 \}$.  Suppose that $t \in \mathcal{M}$ is drawn from the distribution with density proportional to $e^{-q(t)}$ restricted to $\mathcal{M}$.  Then there exists some constant $\Gamma(r,R,\omega,\nu) > 0$ such that 
\[ \left| \mathbf{E} (\tau_{1 m_1} + \tau_{2 m_2} + \ldots + \tau_{\nu m_{\nu}})(\tau_{1 p_1} + \tau_{2 p_2} + \ldots + \tau_{\nu p_{\nu}}) \right| \leq \frac{\Gamma}{k^{\nu-1}} \]
for all $m_1,\ldots,m_{\nu},p_1,\ldots,p_{\nu}$, and
\[ \left| \mathbf{E} (\tau_{1 m_1} + \tau_{2 m_2} + \ldots + \tau_{\nu m_{\nu}})(\tau_{1 p_1} + \tau_{2 p_2} + \ldots + \tau_{\nu p_{\nu}}) \right|\leq \frac{\Gamma}{k^{\nu}}\] 
as long as $m_j \neq p_j$ for all $j=1,\ldots,\nu$.  The constant $\Gamma$ may be chosen to be 
\[\Gamma = \frac{4 \nu^4 R^2}{r^3 \omega^{7\nu-5}}.  \]
\end{theorem} 
We will make use of two basic lemmas.
\begin{lemma} \label{lem:PushForwardInvariance} Let $\mathcal{M}_1$, $\mathcal{M}_2$ be any subspaces of codimension $\nu-1$ such that $\mathcal{M}_1 \cap \ker(q) = \mathcal{M}_2 \cap \ker(q) = \left\{ 0 \right\}$.  Then
\[ \mathbf{E_1} (\tau_{1 m_1} + \ldots + \tau_{\nu m_{\nu}})(\tau_{1 p_1} + \ldots + \tau_{\nu p_{\nu}}) = \mathbf{E_2} (\tau_{1 m_1} + \ldots + \tau_{\nu m_{\nu}})(\tau_{1 p_1} + \ldots + \tau_{\nu p_{\nu}}), \]
where $\mathbf{E_1}$ is taking the expected value over the distribution with density proportional to $e^{-q(t)}$ restricted to $\mathcal{M}_1$, and $\mathbf{E_2}$ the expected value over the distribution with density proportional to $e^{-q(t)}$ restricted to $\mathcal{M}_2$. \end{lemma}
\begin{proof}
Let $S: \mathcal{M}_1 \to \mathcal{M}_2$ be the restriction of the orthogonal projection onto $\mathcal{M}_2$ whose kernel is $\ker(q)$:
\[ St = t+u\ \ \ \text{with}\ \ \ u\in \ker(q)\ \ \ \text{for all}\ \ \ t\in \mathcal{M}_1. \]
As $\det(q|_{\mathcal{M}_1}) \det(S) = \det(q|_{\mathcal{M}_2})$, and $e^{-q(t)} = e^{-q(St)}$, we get that the push forward of the probability measure with density proportional to $e^{-q(t)}$ restricted to $\mathcal{M}_1$ by $S$ is equal to the probability measure with density proportional to $e^{-q(t)}$ restricted to $\mathcal{M}_2$.  Furthermore, $(\tau_{1 m_1} + \ldots + \tau_{\nu m_{\nu}})$ for any $m_1, \ldots, m_{\nu}$ is unchanged when replacing $t$ by $St$.  Therefore
\[ \mathbf{E_1} (\tau_{1 m_1} + \ldots + \tau_{\nu m_{\nu}})(\tau_{1 p_1} + \ldots + \tau_{\nu p_{\nu}}) = \mathbf{E_2} (\tau_{1 m_1} + \ldots + \tau_{\nu m_{\nu}})(\tau_{1 p_1} + \ldots + \tau_{\nu p_{\nu}}) \]
as required. \end{proof}

\begin{lemma}\label{lem:AlmostOrthogonalRemainsThatWay}
Let $v_1, v_2 \in \mathbb{R}^n$, and $C: \mathbb{R}^n \to \mathbb{R}^n$ be a positive definite self-adjoint linear transformation, and let there exist absolute constants $\gamma_1$, $\gamma_2$, $\Gamma_1$, $\Gamma_2$, $\Gamma_3$, and $\Gamma_4$ so that
\begin{enumerate}
\item \[||Cv_i - \gamma_i v_i|| \leq \Gamma_1\ \ \ \text{for } i=1,2,\]
\item \[ \left| \left<Cv_i-\gamma_i v_i, v_j \right> \right| \leq \Gamma_2\ \ \ \text{for } i\neq j, \]
\item \[\left| \left<v_1, v_2 \right> \right| \leq \Gamma_3, \text{ and}\]
\item \[ \lambda_{n}(C) \geq \Gamma_4 \]
\end{enumerate}
where $\lambda_n(C)$ is the smallest eigenvalue of $C$.
Then
\[ \left| \left<C^{-1}v_1, v_2 \right> \right| \leq \frac{\Gamma_3}{\gamma_1} + \frac{\Gamma_2}{\gamma_1 \gamma_2} + \frac{\Gamma_1^2}{\gamma_1 \gamma_2 \Gamma_4}. \]
\end{lemma}
Informally, the lemma states that if $v_1$ and $v_2$ are very close to being orthogonal eigenvectors of $C$, then $C^{-1}v_1$ is close to being orthogonal to $v_2$.
\begin{proof}
We can write this expression as
\[\left|\left<C^{-1}v_1, v_2 \right>\right| = \frac{1}{\gamma_1}\left|\left< C^{-1}\left(\gamma_1 v_1-Cv_1+Cv_1 \right), v_2 \right>\right| .\]
By linearity and the triangle inequality,
\begin{equation}\label{eq:TheFirstStupidTrickForMyStupidProof}\left|\left<C^{-1}v_1, v_2 \right>\right| \leq \frac{1}{\gamma_1}\left|\left<v_1, v_2 \right>\right| + \frac{1}{\gamma_1}\left|\left<C^{-1}\left(\gamma_1 v_1 - Cv_1 \right), v_2 \right>\right| .\end{equation}
Using that $C^{-1}$ is self adjoint, and by linearity and the triangle inequality again we get
\[ \left|\left< C^{-1} \left(\gamma_1 v_1 - Cv_1 \right), v_2 \right>\right| \leq \frac{1}{\gamma_2}\left( \left|\left<\gamma_1 v_1 - Cv_1, v_2 \right>\right| +  \left|\left<\gamma_1 v_1 - Cv_1, C^{-1} \left(\gamma_2 v_2 - Cv_2 \right) \right>\right| \right) .\] 
By conditions (1) and (4), we have
\[ \left| \left<\gamma_1 v_1 - Cv_1, C^{-1} \left(\gamma_2 v_2 - Cv_2 \right) \right> \right| \leq \frac{\Gamma_1^2}{\Gamma_4}. \]
Combining this with condition (2) yields
\[ \left| \left<C^{-1} \left(\gamma_1 v_1 - Cv_1 \right), v_2 \right> \right| \leq \frac{\Gamma_2}{\gamma_2} + \frac{\Gamma_1^2}{\Gamma_4 \gamma_2}.\]
This along with condition (3) and Equation ~\eqref{eq:TheFirstStupidTrickForMyStupidProof} completes the proof.

\end{proof} 
We use this lemma to prove the following:
\begin{lemma}\label{lem:InnerProductBoundForCorrelation}
Let $B_\mathcal{M}$ be the linear transformation $B$ restricted to $\mathcal{M} = \im(B)$, and let $S: \mathbb{R}^{k_1+\ldots +k_{\nu}} \to \mathbb{R}^{k_1+\ldots + k_{\nu}}$ be the orthogonal projection onto $\mathcal{M}$. Then there exist constants $\kappa_1 = \kappa_1(r,\omega,\nu) > 0$ and $\kappa_2 = \kappa_2(r,R,\nu,\omega) > 0$ such that for any choice of $i$ and $j$, we have
\[\left| \left< B_{\mathcal{M}}^{-1} Se_j, Se_i \right> \right|\leq \kappa_1\delta_{ij} k^{-\nu+1} + \kappa_2 k^{-\nu}. \]
If $\nu \geq 3$, the constants may be chosen such that \[ \kappa_1 = \frac{2R}{r^2 \omega^{2\nu-2}},\ \ \ \text{and}\ \ \ \kappa_2 = \frac{7 \nu^2 R^2}{r^3 \omega^{7\nu-5}}.\]
 \end{lemma}
 If $\nu=2$ the same result holds, but the algebraic simplifications to arrive at $\kappa_2$ requires an extra multiplicative factor greater than $7$.
 \begin{proof}
 We apply Lemma ~\ref{lem:AlmostOrthogonalRemainsThatWay}, with $v_1 = Se_j$, $v_2 = Se_i$, and $C=B_{\mathcal{M}}$.  By Lemma \ref{lem:EigenvaluesOfq(t)}, we get condition (4) is satisfied with
 \begin{equation}\label{eq:Gamma4Bound} \lambda_{k_1+\ldots+ k_{\nu}-\nu+1}\left(B_\mathcal{M} \right) \geq \Gamma_4 = \omega^{\nu-1}r k^{\nu-1}.\end{equation}
 By Lemma \ref{lem:InfinityNormOfProjection}, we can write
 \begin{equation}\label{eq:winfinitynorm} Se_j = e_j + w_j,\ \ \ \text{where} \end{equation}
 \[ ||w_j||_{\infty} \leq \frac{1}{\omega^{2\nu - 1}} k^{-1}. \] This gives condition (3) of Lemma \ref{lem:InfinityNormOfProjection},
 \[ \left|\left<Se_j,Se_i \right> \right| \leq \Gamma_3,\ \ \ \text{with} \]
 \[\Gamma_3 = \delta_{ij} +\frac{2}{\omega^{2\nu-1}}k^{-1} + \frac{1}{\omega^{4\nu-2}} k^{-2}.\]
 We can simplify this to be 
 \begin{equation}\label{eq:Gamma3Bound}  \Gamma_3 = \delta_{ij} +\frac{3}{\omega^{4\nu-2}}k^{-1}.\end{equation}
By ~\eqref{eq:GeneralGradientOfq}, we can see that the entries of $B$ are bounded by the following: the diagonal entries are bounded from above by $Rk^{\nu-1}$ and the off-diagonal entries are bounded by $Rk^{\nu-2}$.  Hence,
 \begin{equation}\label{eq:gamma1gamma2definition}B e_j = \gamma_1 e_j + w'_j, \ \ \text{and}\ \ \ B e_i = \gamma_2 e_i + w'_i, \ \ \ \text{where} \end{equation}
\begin{equation} \label{eq:wprimeinfinitynorm} ||w'_j||_{\infty},\ ||w'_i||_{\infty} \leq R k^{\nu-2},\ \ \ \text{and} \end{equation}
\[ r \omega^{\nu-1} k^{\nu-1} \leq \gamma_1,\ \gamma_2 \leq Rk^{\nu-1}.\]

Therefore, using that $B_\mathcal{M} Se_j = B Se_j$ and applying $B$ to ~\eqref{eq:winfinitynorm},
\[ B_\mathcal{M} Se_j  =\gamma_1 e_j + w'_j + Bw_j. \]
Applying ~\eqref{eq:winfinitynorm} again, we get
\[ ||B_\mathcal{M} S e_j - \gamma_1 S e_j || \leq ||w'_j|| + ||Bw_j|| + ||\gamma_1 w_j||, \]
and similarly for $e_i$.
By ~\eqref{eq:wprimeinfinitynorm}, 
\[ ||w'_j||,\ ||w'_i|| \leq \sqrt{\nu k } R k^{\nu-2}, \]
and by ~\eqref{eq:winfinitynorm} and Lemma ~\ref{lem:EigenvaluesOfq(t)},
\[ ||Bw_j||,\ ||Bw_i|| \leq \sqrt{\nu k} \frac{1}{\omega^{2\nu-1}}R\nu k^{\nu-2}. \]
Also, by ~\eqref{eq:winfinitynorm} and ~\eqref{eq:wprimeinfinitynorm},
\[ ||\gamma_1 w_j||,\ ||\gamma_2 w_i|| \leq \frac{\sqrt{\nu k} R}{\omega^{2\nu-1}} k^{\nu-2}. \]
Therefore we get condition (1) is satisfied with
\[ \Gamma_1 =\sqrt{\nu}R\left(1+\frac{2}{\omega^{2\nu-1}}\right)  k^{\nu-1.5}, \]
and $\gamma_1$, $\gamma_2$ as in ~\eqref{eq:gamma1gamma2definition}.  We simplify the bound to be
\begin{equation}\label{eq:Gamma1Bound} \Gamma_1 = \frac{3\sqrt{\nu}R}{\omega^{2\nu-1}} k^{\nu-1.5}. \end{equation}  

Also,
\[ \left|\left<BSe_j-\gamma_1 Se_j,Se_i \right> \right| \leq \]
\[||w'_j||_{\infty} + ||Bw_j||_{\infty}  + \left| \left<w'_j, w_i \right> \right| + \left|\left<Bw_j, w_i \right> \right| + \gamma_1\left<e_j,e_i\right> + \gamma_1 ||w_i||_{\infty}. \]
Using ~\eqref{eq:winfinitynorm} and the bounds on the entries of $B$ described above we get
\[ ||Bw||_{\infty} \leq \frac{(\nu+1)R}{\omega^{2\nu-1}} k^{\nu-2}. \]
This along with ~\eqref{eq:wprimeinfinitynorm}, and the fact that for $u,v\in \mathbb{R}^n$ we have $\left| \left<u,v \right> \right| \leq n ||u||_{\infty} ||v||_{\infty}$, gives that condition (2) holds with
\[\Gamma_2 \leq \]
\[ \left(1+\frac{\nu}{\omega^{2\nu-1}} + \frac{\nu+1}{\omega^{2\nu-1}} + \frac{\nu(\nu+1)}{\omega^{2\nu-1}} + \frac{1}{\omega^{2\nu-1}} \right) R k^{\nu-2} +\delta_{ij} R k^{\nu-1} .\]
We can simplify this bound to be 
\begin{equation}\label{eq:Gamma2Bound}\Gamma_2 = \frac{3\nu^2 R}{\omega^{2\nu-1}} k^{\nu-2} + \delta_{ij} R k^{\nu-1}. \end{equation}
 
Taking ~\eqref{eq:Gamma4Bound}, \eqref{eq:Gamma3Bound}, ~\eqref{eq:Gamma1Bound}, ~\eqref{eq:gamma1gamma2definition}, and ~\eqref{eq:Gamma2Bound}, and applying Lemma ~\ref{lem:AlmostOrthogonalRemainsThatWay} gives us
\[ \left|\left< B_{\mathcal{M}}^{-1} Se_j, Se_i \right> \right| \leq \]
\[
\delta_{ij}\left(\frac{1}{r \omega^{\nu-1}}+\frac{R}{r^2 \omega^{2\nu-2}}\right) k^{-\nu+1} + \left(\frac{3}{r\omega^{5\nu-3}} + \frac{3\nu^2 R }{r^2\omega^{4\nu-3}} + \frac{9\nu R^2}{r^3\omega^{7\nu-5}} \right) k^{-\nu}.\]
Simplifying the coefficients completes the proof.

 \end{proof}
The last observation we need is:
\begin{lemma} \label{lem:GeneralCorrelations} Let $\psi(z):\mathbb{R}^d \to \mathbb{R}$ be a positive definite quadratic form, and let $D$ be the positive definite matrix such that $\psi(z) = \frac{1}{2}\left<z,Dz \right>$.  Let $l_1(z) = \left<v_1, z \right>$ and $l_2(z) = \left<v_2, z \right>$ for some fixed $v_1, v_2 \in \mathbb{R}^d$.  If $z$ is drawn from the distribution with density proportional to $e^{-\psi(z)}$, then
\[ \mathbf{E}(l_1(z) l_2(z)) = \left<v_1, D^{-1} v_2 \right>. \]
\end{lemma}
\begin{proof}
Let
\[ v_i = (v_i^1,\ldots,v_i^d)\ \ \ \text{for}\ \ \ i=1,2. \]
By linearity of expectation we can write
\[ \mathbf{E}(l_1(z) l_2(z)) = \sum_{i,j=1}^{d} v_1^i v_2^j \mathbf{E}(z^i z^j)\ \ \ \text{for}\ \ \ z = (z^1,\ldots,z^d).\]
$D^{-1}$ is exactly the matrix whose entries are $\mathbf{E}(z^i z^j)$, so the sum composes into $\left<v_1, D^{-1} v_2 \right>$ which completes the proof. \end{proof}
We are now ready to prove Theorem ~\ref{thm:Correlations}.  For notational purposes it will be convenient to write $t$ as
 \[ t = (\chi_1,\ldots,\chi_{k_1+k_2+...+k_{\nu}}), \]
 so for example $\tau_{11} = \chi_1$ and $\tau_{22} = \chi_{k_1+2}$.
\begin{proof}
By Lemma ~\ref{lem:PushForwardInvariance}, it suffices to prove the result only for $\mathcal{M} = \ker(q)^{\perp}$.  By Lemma ~\ref{lem:GeneralCorrelations}, if $t$ is drawn from $\mathcal{M}$ with distribution with density proportional to $e^{-q(t)}$ then for $t=(\chi_1,\ldots,\chi_{k_1+k_2+\ldots+k_{\nu}})$,
\[ \mathbf{E}(\chi_i \chi_j) = \left<Se_i, B_\mathcal{M}^{-1} S e_j \right>. \]
We apply Lemma ~\ref{lem:InnerProductBoundForCorrelation}, noting that we can simplify the bound to be
\[ \left| \left< B_{\mathcal{M}}^{-1} Se_j, Se_i \right> \right| \leq \frac{14\nu^2 R^2}{r^3 \omega^{7\nu-5}} k^{-\nu+\delta_{ij}}. \] We distribute and use the linearity of expectation and the triangle inequality to get
\[ \left|\mathbf{E}(\tau_{1 m_1} + \ldots + \tau_{\nu m_{\nu}})(\tau_{1 p_1} + \ldots + \tau_{\nu p_{\nu}}) \right| \leq \sum_{i,j=1}^{\nu} \mathbf{E}\left|\tau_{i m_i} \tau_{j p_j} \right|. \]
In the event that at least one $m_j = p_j$, we can bound all $\nu^2$ terms of the form $\left|\mathbf{E}(\chi_i \chi_l)\right|$ by 
\[ \frac{14 \nu^2 R^2}{r^3 \omega^{7\nu-5}} k^{-\nu+1},\] giving
\[\left|\mathbf{E}(\tau_{1 m_1} + \ldots + \tau_{\nu m_{\nu}})(\tau_{1 p_1} + \ldots + \tau_{\nu p_{\nu}}) \right| \leq \frac{14 \nu^4 R^2}{r^3 \omega^{7\nu-5}} k^{-\nu+1}. \]
If there is no $j$ such that $m_j = p_j$, we can bound each expression of the form $\left|\mathbf{E}(\chi_i \chi_l)\right|$ by 
\[\frac{14 \nu^2 R^2}{r^3 \omega^{7\nu-5}} k^{-\nu},\] giving
\[\left|\mathbf{E}(\tau_{1 m_1} + \ldots + \tau_{\nu m_{\nu}})(\tau_{1 p_1} + \ldots + \tau_{\nu p_{\nu}}) \right| \leq \frac{14 \nu^4 R^2}{r^3 \omega^{7\nu-5}} k^{-\nu}. \]
This completes the proof.
\end{proof}

\section{The Third Degree Term}
\label{s:ThirdDegreeTerm}
In this section we continue to use the notation introduced at the beginning of Section ~\ref{s:Eigenspace}, in particular the quadratic form $q(t)$, the subspace $\mathcal{L}$, and the constants $r$, $R$, $\omega$ and $k$.  The main result of this section is:
\begin{lemma}\label{lem:InsideRegion}
Assume $\nu \geq 3$, and let
 \[ u_{m_1\ldots m_{\nu}} = \beta_{m_1,\ldots,m_{\nu}} (\tau_{1 m_1} + \ldots + \tau_{\nu m_{\nu}}) \]
be random variables for $1\leq m_j \leq k_j$ for each $j=1,\ldots,\nu$, where $t = (\tau_{11}, \tau_{12},\ldots \tau_{\nu k_{\nu}})$ is drawn from the distribution with probability density proportional to $e^{-q(t)}$ restricted to $L$.  Let $\theta > 0$ be chosen such that
\[  \left|\beta_{m_1,\ldots,m_{\nu}}\right| \leq \theta\ \ \ \text{for all } m_1,\ldots, m_{\nu},\]
 and let 
 \[ U = \sum_{m_{1}, \ldots, m_{\nu}=1}^{k_1,\ldots,k_{\nu}} u^3_{m_1 \ldots m_{\nu}}. \]
 Then there exists a constant $\Gamma = \Gamma(\theta,\nu,\omega,R,r) > 0$ such that 
 \[ \left| \mathbf{E} e^{iU} - 1 \right| \leq \Gamma k^{2-\nu}. \]
 The constant $\Gamma$ may be chosen to be
 \[ \Gamma = \frac{3360 \theta^6 \nu^{13} R^6}{r^9 \omega^{21\nu-15}}. \]
\end{lemma}
 We will apply this lemma in the proof of Theorems ~\ref{IntegerPoints} and ~\ref{BinaryPoints}.  In the proof of these theorems, we will show the points to be counted can be expressed as the integral of a function $F(t)$ (different for each theorem).  We will construct a neighborhood, which in the proof will be called $X_1$, of the origin in which we can use Taylor polynomial approximations to express $F$ as $F(t) = e^{-q(t)+if(t)+h(t)}$, where $f(t)$ is a pure cubic function in the form of $U$ in Lemma ~\ref{lem:InsideRegion}, and $h(t)$ is small.  We will also show that the asymptotically all of the integral of $e^{-q(t)}$ is contained in $X_1$. Combining these will allow us to approximate $\int_{X_1} F(t) dt$, and show it is asymptotically equal to $\int_{\mathcal{L}} e^{-q(t)}$.
 
 \
 
 We note that a version of Lemma ~\ref{lem:InsideRegion} holds for $\nu = 2$ as well.  Then the upper bound does not asymptotically go to zero as $k$ goes to infinity, and this qualitative difference causes the theorem for counting integer points in $2$-way transportation polytopes to have an additional factor known as the Edgeworth correction, see \cite{2dcase} for details.

\ 
 
The proof of Lemma ~\ref{lem:InsideRegion} relies on a more general result based on Wick's formula, see for example \cite{WicksFormula}, for the expected value of a product of Gaussian random variables. Let $w_1,\ldots,w_l$ be Gaussian random variables with expected value of $0$.  Then
\[ \mathbf{E}( w_1\ldots w_l) = 0\ \ \ \text{if } l \text{ is odd, and} \]
\[ \mathbf{E} (w_1\ldots w_l) = \sum \left( \mathbf{E} w_{i_1} w_{i_2} \right)\ldots \left( \mathbf{E} w_{i_{l-1}} w_{i_l} \right)\ \ \ \text{if } l \text{ is even}, \]
where the sum is taken over all unordered pairings of the set of indices $1,2,\ldots,l$. In particular,
\begin{equation}\label{w1Cubedw2Cubed} \mathbf{E} w_1^3 w_2^3 = 9 \left( \mathbf{E} w_1^2 \right)\left( \mathbf{E} w_2^2 \right) \left( \mathbf{E} w_1 w_2 \right) + 6\left( \mathbf{E} w_1 w_2 \right)^3. \end{equation}
Note that the random variables $u_{m_1,\ldots,m_{\nu}}$ are Gaussian random variables by construction.  We are now ready to prove Lemma ~\ref{lem:InsideRegion}.
\begin{proof}\label{lem:WicksTheoremApplication} By Theorem ~\ref{thm:Correlations},
\[ \left|\mathbf{E} u_{m_1,\ldots,m_{\nu}} u_{p_1,\ldots,p_{\nu}}\right| \leq \frac{14\theta^2 \nu^4 R^2}{r^3 \omega^{7\nu-5}}k^{-\nu+1}\ \ \  \text{for all}\ \ \ m_1,\ldots,m_{\nu},p_1,\ldots,p_{\nu},\ \ \ \text{and} \]
\[ \left| \mathbf{E} u_{m_1,\ldots,m_{\nu}} u_{p_1,\ldots,p_{\nu}} \right| \leq \frac{14\theta^2 \nu^4 R^2}{r^3 \omega^{7\nu-5}}k^{-\nu}\ \ \ \text{if}\ \ \ m_j \neq p_j\ \ \ \text{for all} \ \ \ 1\leq j\leq \nu. \]

By ~\eqref{w1Cubedw2Cubed}, with $w_1 = u_{m_1,\ldots, m_{\nu}}$ and $w_2 = u_{p_1,\ldots,p_{\nu}}$,
\[ \left| \mathbf{E}(u^3_{m_1\ldots m_{\nu}} u^3_{p_1\ldots p_{\nu}} )\right| \leq   \frac{3360 \theta^6 \nu^{12} R^6}{r^9 \omega^{21\nu-15}}k^{-3\nu+2}\ \ \ \text{for all}\ \ \ m_1,\ldots,m_{\nu},p_1,...,p_{\nu},\ \ \ \text{and} \]
\[\left| \mathbf{E}(u^3_{m_1\ldots m_{\nu}} u^3_{p_1\ldots p_{\nu}})\right| \leq \frac{3360 \theta^6 \nu^{12} R^6}{r^9 \omega^{21\nu-15}}k^{-3\nu+1}\ \ \ \text{if}\ \ \ m_j \neq p_j\ \ \ \text{for}\ \ \ j=1,\ldots,\nu. \]
There are no more than $k^{2\nu}$ total choices of $m_1,\ldots,m_{\nu},p_1,\ldots,p_{\nu}$, and no more than $\nu k^{2\nu-1}$ of them in which there exists $j$ such that a pair $m_j$ and $p_j$ are equal, so
\begin{equation}\label{eq:ValueOfUSquared} \mathbf{E} U^2 \leq  \frac{3360 \theta^6 \nu^{12}(\nu+1) R^6}{r^9 \omega^{21\nu-15}} k^{-\nu+2}. \end{equation}
By the Taylor series estimate
\[\left| e^{i \xi} - (1+i\xi) \right| \leq \frac{1}{2} \xi^2\ \ \ \text{for}\ \ \ \xi \in \mathbb{R},\]
along with the triangle inequality for expected values, we get that
\[ \left|\left( \mathbf{E}e^{iU} \right)-1 \right| \leq \frac{1}{2}\mathbf{E}U^2\ \ \text{yields}\]
\[ \left| \mathbf{E}\left( e^{iU} \right)-1 \right| \leq \frac{1680 \theta^6 \nu^{12}(\nu+1) R^6}{r^9 \omega^{21\nu-15}} k^{-\nu+2}. \]
Applying the simplification $\nu+1 \leq 2\nu$ completes the proof.
 \end{proof}

\section{The Proof of Theorem ~\ref{IntegerPoints}}
\label{s:IntegerPoints}
In this section, we complete the proof of Theorem ~\ref{IntegerPoints}.  For the entirety of this section we use the notation introduced in the statement of the theorem, most importantly the quadratic form $q(t)$ and the constants $r$, $R$, $\omega$ and $k$.  We also recall the overdetermined system of equations for a multi-index transportation polytope of the form $Ax = b$, where $A$ has columns $a_1,\ldots,a_n$ as described in Section ~\ref{ss:ConstraintMatrix}, along with the subspace $\mathcal{L}$ that describes a linearly independent set of equations.  The matrix $Q:\mathbb{R}^{k_1+\ldots+k_{\nu}} \to \mathbb{R}^{k_1+\ldots+k_{\nu}}$ will be the orthogonal projection onto $\mathcal{L}$.  The outline of the proof is as follows: we construct a function $F(t)$, and show that for a multi-index transportation polytope $P$ as in Theorem ~\ref{IntegerPoints},
\[ \left|P\cap \mathbb{Z}^n\right| = \frac{e^{g(z)}}{(2\pi)^{(k_1+\ldots+k_{\nu}-\nu+1)/2}} \int_{\Pi} F(t) dt, \]
where $\Pi = \mathcal{L}\cap[-\pi,\pi]^{k_1+\ldots+k_{\nu}}$.  We then split $\Pi$ up into three regions: an outside region $X_3$, a middle region $X_2$, and an inner region $X_1$.  We show that
\[ \int_{X_2 \cup X_3} F(t) dt\ \ \ \text{and}\ \ \ \int_{\mathcal{L}\setminus X_1 } e^{-q(t)} dt \]
are negligible compared $\int_{\mathcal{L}} e^{-q(t)} dt$.  We show through use of Taylor polynomial approximations that in $X_1$, $F(t) \approx e^{-q(t) + if(t)+h(t)}$, where $h(t)$ is small in $X_1$, and $f(t)$ is a cubic polynomial in $t$ of the form given in Lemma ~\ref{lem:InsideRegion}.  We finish the proof by applying Lemma ~\ref{lem:InsideRegion} to show that 
\[ \int_{X_1} \left|F(t) - e^{-q(t)} \right| dt \ll \int_{\mathcal{L}} e^{-q(t)} dt. \]

\subsection{Integral Expression of the Counting Problem} \label{ss:IntegerPointEntropy}
We use two results of \cite{entropy} to express the number of integer points of $P$ as an integral of a function $F(t)$.  Let $\Pi \subset \mathcal{L}$ be the cube centered at the origin:
\[ \Pi = \{t\in \mathcal{L}\ :\ ||t||_{\infty} \leq \pi \}. \]
We will show that for multi-index transportation polytopes $P$ satisfying the conditions of Theorem ~\ref{IntegerPoints}, the number of integer points satisfies
\begin{equation}\label{eq:IntegralRepresentationIntegerPoints} \left|P\cap \mathbb{Z}^n \right| = \frac{e^{g(z)}}{(2\pi)^{k_1+\ldots+k_{\nu}-\nu+1}} \int_{\Pi} e^{-i\left<t,b\right>} \prod_{j=1}^{n}  \frac{1}{1+\zeta_j-\zeta_j e^{i\left<a_j,t\right>}}dt. \end{equation}
Before we do, we recall the concept of a geometric random variable.  We say $x$ is a geometric random variable if for some $0<p<1$,
\[ \mathbf{Pr}(x = j) = (1-p) p^j\ \ \ \text{for all}\ \ \ j \in \mathbb{Z}_{\geq 0}.\]
In this case, $\mathbf{E}x = \frac{p}{1-p}$.  Conversely, if $\mathbf{E}x = \zeta$, then $p = \frac{\zeta}{1+\zeta}$.  The first theorem we need is the following:
\begin{theorem}\label{thm:IntegerPointEntropy}
Let $P \subset \mathbb{R}^n$ be the intersection of an affine subspace in $\mathbb{R}^n$ and the non-negative orthant $\mathbb{R}^n_+$.  Suppose that $P$ is bounded and has a non-empty interior, that is a point $y = (\eta_1,\ldots,\eta_n)$ where $\eta_i > 0$ for $i=1,\ldots,n$.  Then the strictly concave function
\[ g(x) = \sum_{j=1}^{n} \left( (\xi_j+1)\ln(\xi_j+1) - \xi_j \ln(\xi_j) \right) \]
attains its maximum value in $P$ at a unique point $z = (\zeta_1,\ldots, \zeta_n)$ such that $\zeta_j > 0$ for $j=1,\ldots,n$.  Furthermore, suppose $x_1,\ldots, x_n$ are independent geometric random variables with $\mathbf{E}x_j = \zeta_j$, and let $X = (x_1,\ldots, x_n)$.  Then the probability mass function of $X$ is constant on $P\cap \mathbb{Z}^n$ and equal to $e^{-g(z)}$ at every $x\in P\cap \mathbb{Z}^n$.  In particular,
\[\left|P\cap \mathbb{Z}^n\right| = e^{g(z)}\mathbf{Pr}\left(X\in P \right).\]
\end{theorem}
This is Theorem 4 of \cite{entropy}.  This theorem lets us reduce counting the number of integer points in $P$ to calculating $\mathbf{Pr}\left(X\in P \right)$.  We combine this result with the following:
\begin{lemma}\label{lem:IntegralRepresentationIntegerPoint}Let $p_j,q_j$ be positive numbers such that $p_j+q_j = 1$ for $j=1,\ldots,n$ and let $\mu$ be the geometric measure on the set $\mathbb{Z}^n_+$ of non-negative integer vectors with
\[\mu\{x\} = \prod_{j=1}^{n} p_jq_j^{\xi_j}\ \ \ \text{for}\ \ \ x=(\xi_1,\ldots,\xi_n). \]
Let $P$ be defined by the linear equalities $Ax = b$, where $A$ has columns $a_1,\ldots,a_n$, and\\
$a_1,\ldots,a_n,b\in\mathbb{R}^d$.  Let $\Pi = [-\pi,\pi]^d$ be a cube centered at the origin in $\mathbb{R}^d$.  Then
\[ \mu\left(P\right) = \frac{1}{(2\pi)^d}\int_{\Pi} e^{-i\left<t,b\right>} \prod_{j=1}^{n} \frac{p_j}{1-q_je^{i\left<a_j,t\right>}} dt. \]
Here, $\left<\cdot,\cdot\right>$ is the standard inner product in $\mathbb{R}^d$ and $dt$ is the Lebesgue measure.
\end{lemma}
This is Lemma 13 of \cite{entropy}.  We combine this with Theorem ~\ref{thm:IntegerPointEntropy} to derive ~\eqref{eq:IntegralRepresentationIntegerPoints} in the following way: we identify $\mathcal{L}$ with $\mathbb{R}^{k_1+\ldots+ k_{\nu}-\nu+1}$ in the natural way by identifying the non-zero coordinates of $\mathcal{L}$ with the coordinates of $\mathbb{R}^{k_1+\ldots+k_{\nu}-\nu+1}$.  Then $P$ is defined by the linear equations $QAx = Qb$ where $Q$ is the orthogonal projection onto $\mathcal{L}$.  We note is that for $t\in \mathcal{L}$, we have $\left<b,t\right> = \left<Qb, t\right>$ and $\left<Qa_j,t\right> = \left<a_j,t\right>$ so we can write the integrand using the vectors $a_1,\ldots,a_n,b$ instead of $Qa_1,\ldots,Qa_n,Qb$.  The random variable $X$ in Theorem ~\ref{thm:IntegerPointEntropy} has probability mass function equal to the geometric measure $\mu$ in Lemma ~\ref{lem:IntegralRepresentationIntegerPoint} when $\zeta_j = q_j/p_j = (1-p_j)/p_j$.  This turns the integrand of Lemma ~\ref{lem:IntegralRepresentationIntegerPoint} into 
\[e^{-i\left<t,b\right>}\prod_{j=1}^{n} \frac{1}{1+\zeta_j-\zeta_je^{i\left<a_j,t\right>}}, \]
which proves ~\eqref{eq:IntegralRepresentationIntegerPoints}.
Let 
\[ F(t) = e^{-i\left<t,b\right>}\prod_{j=1}^{n} \frac{1}{1+\zeta_j-\zeta_je^{i\left<a_j,t\right>}}.\]
The bulk of the proof is dedicated to showing that
\[ \int_{\Pi} F(t) dt \approx \int_{\mathcal{L}} e^{-q(t)} dt. \]

\subsection{A Bound on F(t) Away from the Origin}
The main result of this section is the following:
\begin{lemma}\label{lem:OutsideRegionIntegerPoints}Let
\[F(t) = e^{-i\left<t,b\right>} \prod_{j=1}^{n} \frac{1}{1+\zeta_j - \zeta_je^{i\left<a_j,t\right>}}. \]
Then there exists a constant $\gamma = \gamma(\omega,\nu,R) > 0$ such that
\[\left|F(t)\right| \leq \exp\left(-\gamma ||t||_{\infty}^2 k^{\nu-1}\right)\ \ \ \text{for all}\ \ \ t\in \mathcal{L}.\]
If we restrict $t$ such that
\[ \frac{ \omega^{\nu}||t||_{\infty}^2 }{4\pi^2 2^{\nu} \nu^2} k^{\nu-1}\geq 2, \]
then $\gamma$ may be chosen to be
\[ \gamma = \frac{ \omega^{\nu}||t||_{\infty}^2 }{8\pi^2 2^{\nu} \nu^2} \ln\left(1+\frac{2}{5}\pi^2 r\right).\]

\end{lemma}

We apply this lemma in the following way: we construct a region, which in the proof will be called $X_3$, which is the complement of a neighborhood of the origin in $\Pi$.  Then we use Lemma ~\ref{lem:OutsideRegionIntegerPoints} and Lemma ~\ref{lemma:MainVariance} to show that
\[ \left|\int_{X_3\cap \Pi} F(t) dt\right|,\ \int_{X_3} e^{-q(t)} \ll \int_{\mathcal{L}} e^{-q(t)} dt.\]
In $\Pi\setminus X_3$ we will then be able to express $F(t)$ as $F(t) = e^{-q(t)+if(t)+h(t)}$ and show that $f(t)$ and $h(t)$ have a negligible effect on the integral.

To prove Lemma ~\ref{lem:OutsideRegionIntegerPoints} we use the following:
\begin{lemma}\label{lem:OutsideRegionGeneralTheoremIntegerPoints} Let $D$ be a $d\times n$ integer matrix with columns $d_1,\ldots d_n \in \mathbb{Z}^d$.  For each $1\leq l \leq d$, let $Y_l \subset \mathbb{Z}^n$ be a non-empty finite set such that for all $y\in Y_l$, we have $Dy = e_l$, where $e_l$ is the $l$th standard basis vector.  Let $\psi_l:\mathbb{R}^n \to \mathbb{R}$ be the quadratic form
\[ \psi_l(x) = \frac{1}{\left|Y_l \right|} \sum_{y\in Y_l} \left<y,x\right>^2\ \ \ \text{for}\ \ \ x\in \mathbb{R}^n, \]
and let $\rho_l$ be a constant such that 
\[\psi_l(x) \leq \rho_l ||x||^2\ \ \ \text{for all}\ \ \ t\in \mathbb{R}^n.\]  
Suppose further that for $\zeta_1,\ldots,\zeta_n > 0$ we have
\[ \zeta_j+\zeta_j^2 \geq \alpha\ \ \ \text{for some}\ \ \ \alpha > 0\ \ \ \text{and}\ \ \ j=1,\ldots,n.\]
Then for any $t = (\tau_1,\ldots,\tau_d) \in \mathbb{R}^d$, and for each $l$, we have
\[ \left| \prod_{j=1}^{n} \frac{1}{1+\zeta_j - \zeta_j e^{i\left<d_j,t\right>}}\right| \leq \left(1+\frac{2}{5}\alpha \pi^2 \right)^{-\gamma_l}\ \ \ \text{where}\ \ \ \gamma_l = \left\lfloor \frac{\tau_l^2}{\pi^2 \rho_l} \right\rfloor. \] \end{lemma}
This is Lemma 14 of \cite{entropy}.  We are now ready to prove Lemma ~\ref{lem:OutsideRegionIntegerPoints}.  We do so by constructing arrays $Y_l$ for each $e_l \in L$, following a similar construction in a different coordinate system presented in \cite{entropy}.  We then find a uniform bound on $\rho_l$ and apply Lemma ~\ref{lem:OutsideRegionGeneralTheoremIntegerPoints} to every coordinate uniformly.
\begin{proof}
We identify $\mathcal{L}$ with $\mathbb{R}^{k_1+\ldots+k_{\nu}-\nu+1}$ in the natural way, and construct a set $Y_l$ for each $e_l\in L$.  

\

For fixed $1\leq p < k_2$, let $Y_{k_1+p}$ (corresponding to a margin in the second direction) be the set of hypercube arrays labeled with $m_1,m_3,m_4,\ldots,m_{\nu}$ with $1\leq m_j \leq k_j$ for each $j \neq 2$, and let $y_{m_1 m_3 m_4\ldots m_{\nu}}$ that have a $1$ in the $m_1 p m_3\ldots m_{\nu}$ position and a $-1$ in the $m_1 k_2 m_3\ldots m_{\nu}$ position, and a $0$ in every other position.  There are $k'_2$ such arrays, and the corresponding quadratic form is
\[ \psi_{k_1 + p}(x) = \frac{1}{k'_2} \sum_{m,p} \left(\xi_{m_1 p m_2\ldots m_{\nu}} - \xi_{m_1 k_2 m_3\ldots m_{\nu}}\right)^2. \]
No two terms $\left(\xi_{m_1 p m_2\ldots m_{\nu}} - \xi_{m_1 k_2 m_3\ldots m_{\nu}}\right)^2$ of the above sum share any variables, so the eigenvalues of $\psi_{k_1+p}$ are simply the non-zero eigenvalues of the simpler quadratic forms
\[ \frac{1}{k_2'}\left(\xi_{m_1 p m_2\ldots m_{\nu}} - \xi_{m_1 k_2 m_3\ldots m_{\nu}}\right)^2,\]
along with $0$.  The eigenvalues of this quadratic form are $\frac{2}{k'_2}$.  Furthermore, for every $y\in Y_{k_1+p}$, we have $Ay = e_{k_1+p} - e_{k_1+k_2}$, so $QAy = e_{k_1+p}$.

\

Similarly for fixed $1 \leq p < k_3$, let $Y_{k_1+k_2+p}$ (corresponding to a margin in the third direction) be the set of all hypercubes labeled by $m_1 m_2 m_4\ldots m_{\nu}$ with $1\leq m_j \leq k_j$ for all $j\neq 3$, and let $y_{m_1 m_2 p m_4\ldots m_{\nu}}$ have a $1$ in an $m_1 m_2 p m_4\ldots m_{\nu}$ position, a $-1$ in the $m_1 m_2 k_3 m_4\ldots m_{\nu}$ position, and a $0$ in every other position.  There are $k'_3$ such arrays, the corresponding quadratic form has largest eigenvalue $\frac{2}{k'_3}$, and for each $y\in Y_{k_1+k_2+p}$, we have $Ay = e_{k_1+k_2+p} - e_{k_1+k_2+k_3}$, so $QAy = e_{k_1+k_2+p}$.

\

This process can be repeated for any $2 \leq j \leq \nu$ and $1\leq p < k_j$ to get an array $Y_{k_1+\ldots+k_{j-1} + p}$ of hypercubes such that the corresponding quadratic form has largest eigenvalue $\frac{2}{k'_j}$ and for all $y\in Y_{k_1+\ldots+k_{j-1}+p}$, we have $QAy = e_{k_1+\ldots+k_{j-1} + p}$.

\

We now construct $Y_p$ corresponding to any margin in the first direction.  For any choice of $m_1,m_2,m_3,\ldots,m_{\nu}$ with $1\leq m_j < k_j$ for each $2\leq j \leq \nu$ and $1\leq p \leq k_1$ with $m_1 \neq p$, let $y_{m_1 m_2\ldots m_{\nu}}$ be the array which contains a $-(\nu-1)$ in the $m_1 m_2 m_3\ldots m_{\nu}$ position, a $1$ in the $pm_2m_3\ldots m_{\nu}$ position, and a $1$ in every $m_1 m_2\ldots m_{j-1} k_j m_{j+1}...m_{\nu}$ position.  Then the sum over every margin except for the $p$th margin in the first direction and the last margin in every other direction are zero, and the sum over the $p$th margin in the first direction is $1$.  Therefore for all $y\in Y_p$, we have $QAy = e_{p}$ as required. Furthermore there are $\prod_j (k_j-1)$ such points, and the corresponding quadratic form is
\[\displaystyle \psi_p(x) = \prod_{j=1}^{\nu} \frac{1}{k_j-1} \sum_{m_1,\ldots m_{\nu}} \left( -2\xi_{m_1\ldots m_{\nu}} + \xi_{pm_2\ldots m_{\nu}} + \xi_{p k_2m_3 \ldots m_{\nu}} +\ldots+ \xi_{p m_2\ldots m_{\nu-1} k_{\nu}}\right)^2. \]
In general for real numbers $\gamma_1,\ldots, \gamma_{\nu+1}$
\[\left( \sum_{i=1}^{\nu+1} \gamma_i \right)^2 \leq (\nu+1)^2\sum_{i=1}^{\nu+1} \gamma^2_i, \]
so
\[ \psi_{p}(x) \leq \frac{(\nu+1)^2}{\prod_{j=1}^{\nu} (k_j-1)^2} \sum_{m_1\neq p,\ldots,m_{\nu}} 4 \xi_{m_1\ldots m_{\nu}}^2 + \xi_{pm_2\ldots m_{\nu}}^2 + \xi_{pk_2m_3\ldots m_{\nu}}^2 + \ldots + \xi_{pm_2\ldots m_{\nu-1} k_{\nu}}^2. \]
This latter quadratic form has as its eigenvectors the standard unit basis vectors, and the largest eigenvalue it has is bounded by
\[ \frac{4(\nu+1)^2}{\prod_{j=1}^{\nu} (k_j-1)} \max_{j=1,\ldots,\nu}\left\{ k_j-1 \right\}. \]
The subspace $\mathcal{L}$ is spanned by all the standard basis vectors with the exception of $e_{k_1+\ldots+k_{j}}$ for each $j=1,\ldots,\nu$.  For every other $e_l$ we have constructed a set $Y_l$ and a corresponding quadratic form $\psi_l$ with maximum eigenvalues all no larger than 
\[ \frac{4(\nu+1)^2(k-1)}{(\omega k-1)^{\nu}}\] satisfying the hypothesis of Lemma ~\ref{lem:OutsideRegionGeneralTheoremIntegerPoints}.  Furthermore, if $\lambda_l$ is the largest eigenvalue of $\psi_l$ as defined in Section ~\ref{ss:QuadraticFormBasics}, then $\rho_l = \frac{1}{2} \lambda_l$ satisfies the hypothesis of Lemma ~\ref{lem:OutsideRegionGeneralTheoremIntegerPoints}.   Assuming $\omega k - 1 \geq \omega k /2$ and $\nu \geq 3$, we can simplify this bound to
\[ \rho_l = \frac{4 \nu^2 2^{\nu} }{\omega^{\nu}} k^{-\nu+1}. \]
Applying Lemma ~\ref{lem:OutsideRegionGeneralTheoremIntegerPoints} uniformly over all values of $l$ with $D = QA$, and observing we can let $\alpha = r$,  we arrive at
\[ \left|F(t) \right| \leq \left(1+\frac{2}{5}r\pi^2\right)^{-\gamma},\ \ \ \text{where} \]
\[\gamma = \left\lfloor \frac{||t||_{\infty}^2}{\pi^2} \frac{\omega^{\nu}}{4  \nu^2 2^{\nu}}k^{\nu-1} \right\rfloor. \]
As long as 
\[\frac{||t||_{\infty}^2}{\pi^2} \frac{\omega^{\nu}}{4  \nu^2 2^{\nu}} k^{\nu-1} \geq 2,\]
we can apply the inequality
\[\left\lfloor \frac{||t||_{\infty}^2}{\pi^2} \frac{\omega^{\nu}}{4  \nu^22^{\nu}}k^{\nu-1} \right\rfloor \geq \frac{1}{2}\left(\frac{||t||_{\infty}^2}{\pi^2} \frac{\omega^{\nu}}{4  \nu^22^{\nu}} \right) k^{\nu-1}. \]
This completes the proof.
\end{proof}
\subsection{The Proof of Theorem ~\ref{IntegerPoints}}
At this point we are ready to prove Theorem ~\ref{IntegerPoints}.  The outline of the proof is as follows: we first construct a region $X_3 \subset \mathcal{L}$ which is of the form
\[ X_3 = \left\{t\in \mathcal{L}\ :\ ||t||_{\infty} \geq \beta \right\} \]
for some $\beta \in \mathbb{R}$.  We apply Lemma ~\ref{lem:OutsideRegionIntegerPoints} to show that 
\[ \int_{X_3\cap \Pi} |F(t)| dt \ll \int_{\mathcal{L}} e^{-q(t)} dt. \]
For $||t||_{\infty} < \beta$, we express $F(t)$ as
\[ F(t) = e^{-q(t)-if(t)+h(t)},\]
where $q(t)$ is the quadratic form as in Theorem ~\ref{IntegerPoints}, $f(t)$ is a cubic polynomial, and $h(t)$ is bounded by a quartic polynomial.  We use Lemma ~\ref{lemma:MainVariance} along with an inequality comparing $q(t)$ to $h(t)$ to show that for some set $X_2 \subset \mathcal{L}$ of the form
\[ X_2 = \left\{t\in \mathcal{L}\ :\ \delta \leq ||t||_{\infty} \leq \beta \right\}, \]
we have
\[ \int_{X_2} |F(t)| dt \ll \int_{\mathcal{L}} e^{-q(t)} dt. \]
We also use Lemma ~\ref{lemma:MainVariance} to show that
\[ \int_{X_2\cup X_3} e^{-q(t)} dt \ll \int_{\mathcal{L}} e^{-q(t)} dt. \]
We then let
\[ X_1 = \left\{t\in \mathcal{L}\ :\ ||t||_{\infty} \leq \delta \right\}.\]
We show that $|h(t)|$ is small for all $t\in X_1$, and then use Lemma ~\ref{lem:InsideRegion} to show that
\[ \left|\int_{X_1} F(t)-e^{-q(t)} dt\right| \ll \int_{\mathcal{L}}e^{-q(t)} dt. \]
Combining the calculations over the three regions $X_1, X_2,X_3$ will allow us to show that
\[ \int_{\Pi}F(t) dt \approx \int_{\mathcal{L}} e^{-q(t)} dt. \]
\begin{proof}
By ~\eqref{eq:IntegralRepresentationIntegerPoints} and ~\eqref{eq:IntegralOfGaussian}, it suffices to show that
\[ \left|\int_{\mathcal{L}} e^{-q(t)}dt - \int_{\Pi} F(t) \right| \leq \Gamma k^{-\nu+2.5} \]
for some constant $\Gamma$. Let
\begin{equation}\label{eq:DefinitionOfX3IntegerPoints}X_3 = \left\{t\in \mathcal{L}\ :\ ||t||_{\infty}^2 \geq \frac{8 \pi^2 2^{\nu} \nu^2}{\omega^{\nu}\ln\left(1+\frac{2}{5}\pi^2r\right)} \left(\frac{1}{2}\nu^2 k\ln(k) + \frac{1}{2}\nu k \ln(R) \right)k^{-\nu+1} \right\}.\end{equation}
By Lemma ~\ref{lem:OutsideRegionIntegerPoints}, observing that
\[2 \left(\frac{1}{2}\nu^2 k\ln(k) + \frac{1}{2}\nu k \ln(R) \right) \geq 2 \]
always holds as $k\geq 2$, $\nu\geq 3$ and $R\geq 1$,
we have
\[ \int_{X_3 \cap \Pi} |F(t)| dt \leq (2\pi)^{\nu k}\exp\left(-\frac{1}{2}\nu^2 k\ln(k) - \frac{1}{2}\nu k \ln(R) \right).\]
By Corollary ~\ref{cor:BoundOnGaussianIntegral} , we have
\begin{equation}\label{eq:X3IntegerPoints} \int_{X_3\cap \Pi} |F(t)| dt \leq  \exp\left(-\frac{1}{4}\nu^2 k \ln(k)+\nu k \ln(2\pi)\right) \int_{\mathcal{L}} e^{-q(t)} dt, \end{equation}
which is negligible compared to $\int_{\mathcal{L}} e^{-q(t)} dt$.

\

For the middle and inside regions, we can use the Taylor polynomial estimate 
\[\left| e^{i\xi} - 1 - i\xi + \frac{\xi^2}{2} + i \frac{\xi^3}{6} \right| \leq \frac{\xi^4}{24}\ \ \text{ for all }\ \ \xi \in \mathbb{R} \]
to write
\[ e^{i \left<a_j,t \right> } = 1 + i\left<a_j, t \right> - \frac{ \left<a_j,t \right>^2}{2} - i \frac{\left<a_j,t \right>^3}{6} + g_j(t) \left<a_j, t \right>^4, \]
where $|g_j(t)| \leq \frac{1}{24}$ for all $j=1,\ldots,n$ for $n=k_1\times k_2\times \ldots\times k_{\nu}$.  Therefore
\[ F(t) = e^{-i\left<b,t \right>} \prod_{j=1}^{n} \left(1 -\zeta_j + i\zeta_j\left<a_j, t \right> - \zeta_j\frac{ \left<a_j,t \right>^2}{2} - i\zeta_j\frac{\left<a_j,t \right>^3}{6} + \zeta_j g_j(t) \left<a_j, t \right>^4 \right)^{-1}. \]
Furthermore, using
\[\left| \ln(1+\xi) - \xi + \frac{\xi^2}{2} - \frac{\xi^3}{3} \right| \leq \frac{ \left|\xi \right|^4}{2}\ \ \ \text{for all complex}\ \ \ \left|\xi\right| \leq 1/2, \]
plus
\[ \sum_{j=1}^{n} \zeta_j a_j = b,\]
we can write
\[ F(t) = e^{-q(t)-if(t) + h(t)},\ \ \text{ where} \]
\[q(t) = \frac{1}{2} \sum_{m_1,\ldots,m_{\nu}} \left(\zeta_{m_1\ldots m_{\nu}}^2+\zeta_{m_1\ldots m_{\nu}} \right) \left(\tau_{m_1 1} + \tau_{m_2 2} + \ldots + \tau_{m_{\nu} \nu} \right)^2, \]
\[ f(t) = \frac{1}{6} \sum_{m_1,\ldots,m_{\nu}} \left( \zeta_{m_1,\ldots,m_{\nu}}+\zeta_{m_1,\ldots,m_{\nu}}^2 \right)\left(2\zeta_{m_1,\ldots,m_{\nu}}+1\right)\left(\tau_{m_1 1} + \tau_{m_2 2} + \ldots + \tau_{m_{\nu} \nu} \right)^3, \]
is a cubic polynomial of the form in Section ~\ref{s:ThirdDegreeTerm}, and
\begin{equation}\label{eq:GBound} \left| h(t) \right| \leq 2 \sum_{m_1 \ldots m_{\nu}} \left(1+\zeta_{m_1 \ldots m_{\nu}}^4 \right) \left(\tau_{m_1 1} + \tau_{m_2 2} + \ldots + \tau_{m_{\nu} \nu}\right)^4. \end{equation}
This expansion is valid as long as $||t||_{\infty} \leq 1/(2\nu \sqrt{R})$. For $t\in \Pi\setminus X_3$, this inequality is true as long as
\[\frac{8 \pi^2 2^{\nu} \nu^2}{\omega^{\nu}\ln\left(1+\frac{2}{5}\pi^2r\right)} \left(\frac{1}{2}\nu^2 k\ln(k) + \frac{1}{2}\nu k \ln(R) \right)k^{-\nu+1} \leq \frac{1}{4\nu^2 R}, \]
which is assumed by hypothesis.  Let
\[ X_2 = \left\{t\in \mathcal{L}\ :\ \frac{8 \pi^2 2^{\nu} \nu^4}{\omega^{\nu}}  k^{-\nu+1.25} \leq  ||t||_{\infty}^2 \leq  \frac{8 \pi^2 2^{\nu} \nu^4 R}{\omega^{\nu} r}  \ln(k)k^{-\nu+2} \right\}.\]
We have simplified the upper bound on $||t||_{\infty}$ from the $X_3$ lower bound by making it strictly larger, using $\ln(R)/\ln(1+2\pi^2 r/5) \leq R/r $, and $\nu \ln(k) \geq 2$ as long as $k\geq 2$ and $\nu \geq 3$.  Then we get
\[ \left|\int_{X_2}F(t) dt  \right|\leq \int_{X_2} |F(t)| dt = \int_{X_2} e^{-q(t)+h(t)} dt. \]
As 
\[(\tau_{m_1 1}+\ldots + \tau_{m_{\nu} \nu})^2 \leq \frac{8 \pi^2 2^{\nu} \nu^6 R}{\omega^{\nu}r} \ln(k) k^{-\nu+2}\ \ \ \text{for}\ \ \ t\in X_2,\ \ \ \text{and}\]
\[\frac{1+\zeta_{m_1\ldots m_{\nu}}^4}{\zeta_{m_1\ldots m_{\nu}}^2+\zeta_{m_1\ldots m_{\nu}}} \leq \frac{R^2+1}{R} \leq 2R\ \ \ \text{as}\ \ \ R\geq 1,\]
we get for $t\in X_2$ that
\[ |h(t)| \leq \frac{64 \pi^2 2^{\nu} \nu^6 R^2}{\omega^{\nu}r}\ln(k) k^{-\nu+2}q(t).\]
Assuming as in the hypothesis of Theorem ~\ref{IntegerPoints} that
\[\delta = \frac{64 \pi^2 2^{\nu} \nu^6 R^2}{\omega^{\nu}r}\ln(k) k^{-\nu+2} \leq 3/4\]
for $t\in X_2$, we get $|F(t)|= e^{-q(t)+h(t)} \leq e^{-(1-\delta)q(t)}$.  Therefore,
\[\left|\int_{X_2} F(t) dt\right| \leq \int_{X_2} e^{-(1-\delta)q(t)} dt. \]
Doing the change of variables $t\mapsto (\sqrt{1-\delta})t$ we get
\[\left|\int_{X_2} F(t) dt\right| \leq (1-\delta)^{-\nu k/2} \int_{\sqrt{1-\delta}X_2} e^{-q(t)} dt. \]
We use the bound
\[\left(1-\frac{64 \pi^2 2^{\nu} \nu^6 R^2}{\omega^{\nu}r}\ln(k) k^{-\nu+2} \right)^{-\nu k/2} \leq \exp\left(\frac{128 \pi^2 2^{\nu} \nu^5 R^2}{\omega^{\nu}r}\ln(k) k^{-\nu+3} \right), \]
and by Lemma ~\ref{lemma:MainVariance} and the choice of the lower bound in the definition of $X_2$, we get
\begin{equation}\label{eq:MiddleRegionNegligibleIntegerPoint} \left|\int_{X_2} F(t) dt\right| \leq \nu k  \exp\left(\frac{128 \pi^2 2^{\nu} \nu^5 R^2}{\omega^{\nu}r}\ln(k) k^{-\nu+3}-\frac{2\pi^2\omega^{5\nu-3}2^{\nu}r^2}{ R} k^{.25} \right) \int_{\mathcal{L}}e^{-q(t)} dt. \end{equation}

Similarly, by Lemma ~\ref{lemma:MainVariance}, we get
\begin{equation}\label{eq:OutsideGaussianNegligibleIntegerPoint} \int_{X_2\cup X_3} e^{-q(t)}dt \leq \nu k \exp\left(-\frac{8\pi^2\omega^{5\nu-3}2^{\nu}r^2}{2 R} k^{.25} \right)\int_{\mathcal{L}} e^{-q(t)} dt.\end{equation}
For the inner region, we define 
\[ X_1 = \left\{ t \in \mathcal{L}\ :\ ||t||_{\infty}^2 \leq \frac{8 \pi^2 2^{\nu} \nu^4}{\omega^{\nu}} k^{-\nu+1.25} \right\}.\]
For $t\in X_1$, the inequality $\left|\left<a_j, t \right>\right|^4 \leq \nu^4 ||t||_{\infty}^4$ gives us
\[ |h(t)| \leq 2R^2 \frac{64 \pi^4 4^{\nu} \nu^8}{\omega^{2\nu}} k^{-\nu+2.5}. \]
Hence, writing
\[ \left|\int_{X_1} F(t) - e^{-q(t)} dt \right| =  \int_{X_1}\left| e^{-q(t)+if(t) + h(t)} - e^{-q(t)}\right| dt, \]
we get
\[\left|\int_{X_1} F(t) - e^{-q(t)} dt \right| \leq \left(\exp\left(2R^2 \frac{64 \pi^4 4^{\nu} \nu^8}{\omega^{2\nu}} k^{-\nu+2.5}\right)-1\right) \left|\int_{X_1} e^{-q(t)+if(t)} - e^{-q(t)} dt \right|.\]
Applying Lemma ~\ref{lem:InsideRegion} with $\beta^3_{m_1,\ldots,m_{\nu}} = (\zeta_{m_1,\ldots,m_{\nu}}^2+\zeta_{m_1,\ldots,m_{\nu}})(2\zeta_{m_1,\ldots,m_{\nu}}+1) \leq 2R^{3/2}$, and noting that almost all the measure of $e^{-q(t)}$ is contained in $X_1$ by ~\eqref{eq:OutsideGaussianNegligibleIntegerPoint}, we get
\begin{equation}\label{eq:InsideRegionIntegerPoint}\left|\int_{X_1} F(t) - e^{-q(t)} dt\right| \leq\left(\exp\left(2R^2 \frac{64 \pi^4 4^{\nu} \nu^8}{\omega^{2\nu}} k^{-\nu+2.5}\right)-1\right)\times \end{equation}
\[ \left(2 \nu k \exp\left(-\frac{8\pi^2 \omega^{5\nu-3}2^{\nu}r^2}{2 R} k^{.25}\right) + 1+\frac{13440\nu^{13} R^9}{\omega^{21\nu-15}r^9}k^{2-\nu}\right)\int_{\mathcal{L}} e^{-q(t)} dt. \]
Combining Equations ~\eqref{eq:X3IntegerPoints},~\eqref{eq:MiddleRegionNegligibleIntegerPoint},~\eqref{eq:OutsideGaussianNegligibleIntegerPoint}, and ~\eqref{eq:InsideRegionIntegerPoint} completes the proof.  If $k$ is large enough, the $k^{-2.5+\nu}$ term from ~\eqref{eq:InsideRegionIntegerPoint} dominates, and doubling it gives us the example value for $\Gamma$. 
\end{proof}

\section{The Proof of Theorem ~\ref{BinaryPoints}}
\label{s:BinaryPoints}
In this section, we complete the proof of Theorem ~\ref{BinaryPoints}. For the entirety of this section we use the notation introduced in the statement of the theorem, most importantly the quadratic form $q(t)$ and the constants $r$, $R$, $\omega$ and $k$.  We also recall the overdetermined system of equations for a multi-index transportation polytope of the form $Ax = b$, where $A$ has columns $a_1,\ldots,a_n$ as described in Section ~\ref{ss:ConstraintMatrix}, along with the subspace $\mathcal{L}$ that describes a linearly independent set of equations.  The matrix $Q:\mathbb{R}^{k_1+\ldots+k_{\nu}} \to \mathbb{R}^{k_1+\ldots+k_{\nu}}$ will be the orthogonal projection onto $\mathcal{L}$.  The outline of the proof is as follows: we construct a function $F(t)$, and show that for a multi-index transportation polytope $P$ as in Theorem ~\ref{BinaryPoints},
\[ \left|P\cap \{0,1\}^n\right| = \frac{e^{g(z)}}{(2\pi)^{(k_1+\ldots+k_{\nu}-\nu+1)/2}} \int_{\Pi} F(t) dt, \]
where $\Pi \subset \mathcal{L}$ is the set $\{t \in \mathcal{L}\ :\ ||t||_{\infty} \leq \pi \}$.  We then split $\Pi$ up into three regions: an outside region $X_3$, a middle region $X_2$, and an inner region $X_1$.  We show that
\[ \int_{X_2 \cup X_3} F(t) dt\ \ \ \text{and}\ \ \ \int_{\mathcal{L}\setminus X_1 } e^{-q(t)} dt \]
are negligible compared $\int_{\mathcal{L}} e^{-q(t)} dt$.  We show through use of Taylor polynomial approximations that in $X_1$, $F(t) \approx e^{-q(t) + if(t)+h(t)}$, where $h(t)$ is small in $X_1$, and $f(t)$ is a cubic polynomial in $t$ of the form given in Lemma ~\ref{lem:InsideRegion}.  We finish the proof by applying Lemma ~\ref{lem:InsideRegion} to show that 
\[ \int_{X_1} \left|F(t) - e^{-q(t)} \right| dt \ll \int_{\mathcal{L}} e^{-q(t)} dt. \]

\subsection{Integral Expression of the Counting Problem}\label{ss:BinaryPointEntropy}
We use two results of \cite{entropy} to express the number of binary integer points of $P$ as an integral of a function $F(t)$.  Let $\Pi \subset \mathcal{L}$ be the cube centered at the origin:
\[ \Pi = \{t\in \mathcal{L}\ :\ ||t||_{\infty} \leq \pi \}. \]
We will show that for multi-index transportation polytopes $P$ satisfying the conditions of Theorem ~\ref{BinaryPoints}, the number of binary integer points satisfies
\begin{equation}\label{eq:IntegralRepresentationBinaryPoints} \left|P\cap \{0,1\}^n \right| = \frac{e^{g(z)}}{(2\pi)^{k_1+\ldots+k_{\nu}-\nu+1}} \int_{\Pi} e^{-i\left<t,b\right>} \prod_{j=1}^{n}\left(1-\zeta_j+\zeta_j e^{i\left<a_j,t\right>}\right)dt. \end{equation}
Before we do, we recall the concept of a Bernoulli random variable.  We say $x$ is a Bernoulli random variable if for some $0<p<1$,
\[ \mathbf{Pr}(x = 0) = p\ \ \  \text{and}\ \ \ \mathbf{Pr}(x=1) = (1-p).\]
In this case, $\mathbf{E}x = 1-p$.  Conversely, if $\mathbf{E}x = \zeta$, then $p = 1-\zeta$.  The first theorem we need is the following:
\begin{theorem}\label{thm:BinaryPointEntropy}
Let $P \subset \mathbb{R}^n$ be the intersection of an affine subspace in $\mathbb{R}^n$ and the unit cube $[0,1]^n$.  Suppose $P$ is bounded and has a non-empty interior, that is a point $y = (\eta_1,\ldots,\eta_n)$ where $\eta_i > 0$ for $i=1,\ldots,n$.  Then the strictly concave function
\[ g(x) = \sum_{j=1}^{n} \left( \xi_j\ln\frac{1}{\xi_j} + (1-\xi_j) \ln\frac{1}{1-\xi_j} \right) \]
attains its maximum value in $P$ at a unique point $z = (\zeta_1,\ldots, \zeta_n)$ such that $0< \zeta_j <1$ for $j=1,\ldots,n$.  Furthermore, suppose $x_1,\ldots, x_n$ are independent Bernoulli random variables with $\mathbf{E}x_j = \zeta_j$, and let $X = (x_1,\ldots, x_n)$.  Then the probability mass function of $X$ is constant on $P\cap \{0,1\}^n$ and equal to $e^{-g(z)}$ at every $x\in P\cap \{0,1\}^n$.  In particular,
\[\left|P\cap \{0,1\}^n\right| = e^{g(z)}\mathbf{Pr}\left(X\in P \right).\]
\end{theorem}
This is Theorem 5 of \cite{entropy}.  This lets us reduce counting the number of binary integer points in $P$ to calculating $\mathbf{Pr}\left(X\in P \right)$.  We combine this result with the following:
\begin{lemma}\label{lem:IntegralRepresentationBinaryPoint}Let $p_j,q_j$ be positive numbers such that $p_j+q_j = 1$ for $j=1,\ldots,n$, and let $\mu$ be the Bernoulli measure on the set $\{0,1\}^n$ of non-negative integer vectors with
\[\mu\{x\} = \prod_{j=1}^{n} p_j^{1-\xi_j}q_j^{\xi_j}\ \ \ \text{for}\ \ \ x=(\xi_1,\ldots,\xi_n). \]
Let $P$ be defined by the linear equalities $Ax = b$, where $A$ has columns $a_1,\ldots,a_n$, such that $a_1,\ldots,a_n,b\in\mathbb{R}^d$.  Let $\Pi = [-\pi,\pi]^d$ be a cube centered at the origin in $\mathbb{R}^d$.  Then
\[ \mu\left(P\right) = \frac{1}{(2\pi)^d}\int_{\Pi} e^{-i\left<t,b\right>} \prod_{j=1}^{n} \left(p_j+q_je^{i\left<a_j,t\right>}\right) dt. \]
Here, $\left<\cdot,\cdot\right>$ is the standard inner product in $\mathbb{R}^d$ and $dt$ is the Lebesgue measure.
\end{lemma}
This is Lemma 11 of \cite{entropy}.  We combine this with Theorem ~\ref{thm:BinaryPointEntropy} to derive ~\eqref{eq:IntegralRepresentationBinaryPoints} as follows: we identify $\mathcal{L}$ with $\mathbb{R}^{k_1+\ldots+ k_{\nu}-\nu+1}$ in the natural way by identifying the non-zero coordinates of $\mathcal{L}$ with the coordinates of $\mathbb{R}^{k_1+\ldots+k_{\nu}-\nu+1}$.  Then $P$ is defined by the linear equations $QAx = Qb$ where $Q$ is the orthogonal projection onto $\mathcal{L}$.  As $\left<Qa_j,t\right> = \left<a_j,t\right>$ and $\left<Qb,t\right> = \left<b,t\right>$ for $t\in \mathcal{L}$, we use the columns of $A$ and the vector $b$ in the integrand instead of $QA$ and $Qb$.  The random variable $X$ in Theorem ~\ref{thm:BinaryPointEntropy} induces the Bernoulli measure $\mu$ in Lemma ~\ref{lem:IntegralRepresentationBinaryPoint} when $\zeta_j =1-p_j$.  This turns the integrand of Lemma ~\ref{lem:IntegralRepresentationBinaryPoint} into 
\[e^{-i\left<t,b\right>}\prod_{j=1}^{n} \left(1-\zeta_j+\zeta_je^{i\left<a_j,t\right>}\right). \]
Let 
\[ F(t) = e^{-i\left<t,b\right>}\prod_{j=1}^{n} \left(1-\zeta_j+\zeta_je^{i\left<a_j,t\right>}\right).\]
The bulk of the proof is dedicated to showing that
\[ \int_{\Pi} F(t) dt \approx \int_{\mathcal{L}} e^{-q(t)} dt. \]

\subsection{A Bound on F(t) Away from the Origin}
The main result of this section is the following:
\begin{lemma}\label{lem:OutsideRegionBinaryPoints}Let
\[F(t) = e^{-i\left<t,b\right>}\prod_{j=1}^{n} \left(1-\zeta_j+\zeta_je^{i\left<a_j,t\right>}\right). \]
Then there exists a constant $\gamma = \gamma(\omega,\nu,r) > 0$ such that
\[\left|F(t)\right| \leq \exp\left(-\gamma ||t||_{\infty}^2 k^{\nu-1}\right).\]
The constant $\gamma$ may be chosen to be
\[ \gamma = \frac{r\omega^{\nu}}{20  \nu^22^{\nu}}.\]

\end{lemma}

We apply this lemma in the following way: we construct a region, which in the proof will be called $X_3$, which is the complement of a neighborhood of the origin in $\Pi$.  Then we use Lemma ~\ref{lem:OutsideRegionIntegerPoints} and Lemma ~\ref{lemma:MainVariance} to show that
\[ \left|\int_{X_3\cap \Pi} F(t) dt\right|,\ \int_{X_3} e^{-q(t)} dt  \ll \int_{\mathcal{L}} e^{-q(t)} dt.\]

To prove Lemma ~\ref{lem:OutsideRegionIntegerPoints} we use the following:
\begin{lemma}\label{lem:OutsideRegionGeneralTheoremBinaryPoints} Let $D$ be a $d\times n$ integer matrix with columns $d_1,\ldots d_n \in \mathbb{Z}^d$.  For each $1\leq l \leq d$, let $Y_l \subset \mathbb{Z}^d$ be a non-empty finite set such that for all $y\in Y_l$, we have $Dy = e_l$, where $e_l$ is the $l$th standard basis vector.  Let $\psi_l:\mathbb{R}^n \to \mathbb{R}$ be the quadratic form
\[ \psi_l(x) = \frac{1}{\left|Y_l \right|} \sum_{y\in Y_l} \left<y,x\right>^2\ \ \ \text{for}\ \ \ x\in \mathbb{R}^n, \]
and let $\rho_l$ be a constant such that 
\[\psi_l(x) \leq \rho_l ||x||^2\ \ \ \text{for all}\ \ \ x\in \mathbb{R}^n.\]  
Suppose further that for $\zeta_1,\ldots,\zeta_n > 0$ we have
\[ \zeta_j-\zeta_j^2 \geq \alpha\ \ \ \text{for some}\ \ \ \alpha > 0\ \ \ \text{and}\ \ \ j=1,\ldots,n.\]
Then for any $t = (\tau_1,\ldots,\tau_d) \in \mathbb{R}^d$, and for each $l$, we have
\[ \left| \prod_{j=1}^{n} \left(1-\zeta_j + \zeta_j e^{i\left<d_j,t\right>}\right)\right| \leq \exp\left(-\frac{\alpha \tau_l^2}{5\rho_l}\right).\] \end{lemma}
This is Lemma 12 of \cite{entropy}.  We are now ready to prove Lemma ~\ref{lem:OutsideRegionBinaryPoints}.
\begin{proof}
We identify $\mathcal{L}$ with $\mathbb{R}^{k_1+\ldots+k_{\nu}-\nu+1}$ in the natural way.  We use the sets $Y_l$ constructed in the proof of ~\ref{lem:OutsideRegionIntegerPoints}, to get sets $Y_l$ satisfying the hypothesis with
\[ \rho_l = \frac{4  \nu^2 2^{\nu}}{\omega^{\nu}} k^{-\nu+1}. \]
Applying Lemma ~\ref{lem:OutsideRegionGeneralTheoremBinaryPoints} uniformly over all values of $l$ with $D = QA$ and $\alpha = r$,  we arrive at
\[ \left|F(t) \right| \leq \exp\left(-\frac{r\omega^{\nu} ||t||_{\infty}^2}{20  \nu^2 2^{\nu}} k^{\nu-1}\right). \]

\end{proof}
\subsection{The Proof of Theorem ~\ref{BinaryPoints}}
At this point we are ready to prove Theorem ~\ref{BinaryPoints}.  The outline of the proof is as follows: we first construct a region $X_3 \subset \mathcal{L}$ which is of the form
\[ X_3 = \left\{t\in \mathcal{L}\ :\ ||t||_{\infty} \geq \beta \right\} \]
for some $\beta \in \mathbb{R}$.  We apply Lemma ~\ref{lem:OutsideRegionBinaryPoints} to show that 
\[ \int_{X_3\cap \Pi} |F(t)| dt \ll \int_{\mathcal{L}} e^{-q(t)} dt. \]
For $||t||_{\infty} < \beta$, we express $F(t)$ as
\[ F(t) = e^{-q(t)+if(t)+h(t)},\]
where $q(t)$ is the quadratic form as in Theorem ~\ref{BinaryPoints}, $f(t)$ is a cubic polynomial, and $h(t)$ is bounded by a quartic polynomial.  We use Lemma ~\ref{lemma:MainVariance} along with an inequality comparing $q(t)$ to $h(t)$ to show that for some set $X_2 \subset L$ of the form
\[ X_2 = \left\{t\in \mathcal{L}\ :\ \delta \leq ||t||_{\infty} \leq \beta \right\}, \]
we have
\[ \int_{X_2} |F(t)| dt \ll \int_{\mathcal{L}} e^{-q(t)} dt. \]
We also use Lemma ~\ref{lemma:MainVariance} to show that
\[ \int_{X_2\cup X_3} e^{-q(t)} dt \ll \int_{\mathcal{L}} e^{-q(t)} dt. \]
We then let
\[ X_1 = \left\{t\in \mathcal{L}\ :\ ||t||_{\infty} \leq \delta \right\}.\]
We show that $|h(t)|$ is small for all $t\in X_1$, and then use Lemma ~\ref{lem:InsideRegion} to show that
\[ \left|\int_{X_1} F(t)-e^{-q(t)} dt\right| \ll \int_{\mathcal{L}}e^{-q(t)} dt. \]
Combining the calculations over the three regions $X_1, X_2,X_3$ will allow us to show that
\[ \int_{\Pi}F(t) dt \approx \int_{\mathcal{L}} e^{-q(t)} dt. \]
We observe that for all the calculations in Sections ~\ref{s:Eigenspace} through ~\ref{s:ThirdDegreeTerm}, we can replace $R$ with $1$ as $\zeta_{m_1,\ldots,m_{\nu}} - \zeta_{m_1,\ldots,m_{\nu}}^2 \leq 1/4$ always. 
\begin{proof}
By ~\eqref{eq:IntegralRepresentationBinaryPoints} and ~\eqref{eq:IntegralOfGaussian}, it suffices to show
\[ \left| \int_{\Pi} F(t) dt - \int_{\mathcal{L}} e^{-q(t)} dt \right| \leq \Gamma k^{-\nu+2.5} \]
for some constant $\Gamma > 0$.  Let
\begin{equation}\label{eq:DefinitionOfX3BinaryrPoints}X_3 = \left\{t\in \mathcal{L}\ :\ ||t||_{\infty}^2 \geq \frac{10 \nu^2  2^{\nu}}{r\omega^{\nu} } \nu^2 \ln(k)k^{-\nu+2} \right\}.\end{equation}
By Lemma ~\ref{lem:OutsideRegionBinaryPoints}, 
we have
\[ \int_{X_3 \cap \Pi} |F(t)| dt \leq (2\pi)^{\nu k}\exp\left(-\frac{1}{2}\nu^2 k\ln(k) \right).\]
By Corollary ~\ref{cor:BoundOnGaussianIntegral} , we have
\begin{equation}\label{eq:X3BinaryPoints} \int_{X_3\cap \Pi} |F(t)| dt \leq  \exp\left(-\frac{1}{4}\nu^2 k \ln(k)+\nu k \ln(2\pi)\right) \int_{\mathcal{L}} e^{-q(t)} dt, \end{equation}
which is negligible compared to $\int_{\mathcal{L}} e^{-q(t)} dt$.

\

For the middle and inside regions, we can use the Taylor polynomial estimate 
\[\left| e^{i\xi} - 1 - i\xi + \frac{\xi^2}{2} + i \frac{\xi^3}{6} \right| \leq \frac{\xi^4}{24}\ \ \text{ for all }\ \ \xi \in \mathbb{R} \]
to write
\[ e^{i \left<a_j,t \right> } = 1 + i\left<a_j, t \right> - \frac{ \left<a_j,t \right>^2}{2} - i \frac{\left<a_j,t \right>^3}{6} + g_j(t) \left<a_j, t \right>^4, \]
where $|g_j(t)| \leq \frac{1}{24}$ for all $j=1,\ldots,k_1k_2\ldots k_{\nu}$.  Therefore
\[ F(t) = e^{-i\left<b,t \right>} \prod_{j=1}^{n} \left(1 + i\zeta_j\left<a_j, t \right> - \zeta_j\frac{ \left<a_j,t \right>^2}{2} - i\zeta_j\frac{\left<a_j,t \right>^3}{6} + \zeta_j g_j(t) \left<a_j, t \right>^4 \right). \]
Furthermore, using
\[\left| \ln(1+\xi) - \xi + \frac{\xi^2}{2} - \frac{\xi^3}{3} \right| \leq \frac{ \left|\xi \right|^4}{2}\ \ \ \text{for all complex}\ \ \  \left|\xi\right| \leq 1/2, \]
plus
\[ \sum_{j=1}^{n} \zeta_j a_j = b,\]
we can write
\[ F(t) = e^{-q(t)+if(t) + h(t)},\ \ \text{ where} \]
\[q(t) = \frac{1}{2} \sum_{m_1,\ldots,m_{\nu}} \left(\zeta_{m_1\ldots m_{\nu}}-\zeta_{m_1\ldots m_{\nu}}^2 \right) \left(\tau_{m_1 1} + \tau_{m_2 2} + \ldots + \tau_{m_{\nu} \nu} \right)^2, \]
\[ f(t) = \frac{1}{6} \sum_{m_1,\ldots,m_{\nu}} \left( \zeta_{m_1,\ldots,m_{\nu}}-\zeta_{m_1,\ldots,m_{\nu}}^2 \right)\left(2\zeta_{m_1,\ldots,m_{\nu}}-1\right)\left(\tau_{m_1 1} + \tau_{m_2 2} + \ldots + \tau_{m_{\nu} \nu} \right)^3, \]
is a cubic polynomial of the form in Section ~\ref{s:ThirdDegreeTerm}, and
\begin{equation}\label{eq:GBoundBinaryPoints} \left| h(t) \right| \leq 2 \sum_{m_1 \ldots m_{\nu}} \left(\tau_{m_1 1} + \tau_{m_2 2} + \ldots + \tau_{m_{\nu} \nu}\right)^4. \end{equation}
This representation is valid as long as $||t||_{\infty} \leq 1/(2\nu)$. For $t\in \Pi\setminus X_3$, this inequality is true as long as
\[ \frac{10 \nu^42^{\nu}}{r\omega^{\nu}} \ln(k) k^{-\nu+2} \leq \frac{1}{4\nu^2}, \]
 which is assumed by hypothesis.  Let
\[ X_2 = \left\{t\in \mathcal{L}\ :\  \frac{10 \nu^42^{\nu}}{r\omega^{\nu}}  k^{-\nu+1.25}  \leq  ||t||_{\infty}^2 \leq   \frac{10 \nu^42^{\nu}}{r\omega^{\nu}} \ln(k) k^{-\nu+2}  \right\}.\]
Then we get
\[ \left|\int_{X_2}F(t) dt  \right|\leq \int_{X_2} |F(t)| dt = \int_{X_2} e^{-q(t)+h(t)} dt. \]
As 
\[(\tau_{m_1 1}+\ldots + \tau_{m_{\nu} \nu})^2 \leq \nu^2 ||t||_{\infty}^2,\]
we get for $t\in X_2$ that
\[ |h(t)| \leq \frac{20 \nu^6 2^{\nu}}{r^2\omega^{\nu}} \ln(k) k^{-\nu+2} q(t).\]
Assuming as in the hypothesis of Theorem ~\ref{BinaryPoints} that
\[\delta = \frac{20 \nu^6 2^{\nu}}{r^2\omega^{\nu}} \ln(k) k^{-\nu+2} \leq 3/4,\]
for $t\in X_2$ we get $|F(t)|= e^{-q(t)+h(t)} \leq e^{-(1-\delta)q(t)}$.  Therefore,
\[\left|\int_{X_2} F(t) dt\right| \leq \int_{X_2} e^{-(1-\delta)q(t)} dt. \]
Doing the change of variables $t\mapsto (\sqrt{1-\delta})t$ we get
\[\left|\int_{X_2} F(t) dt\right| \leq (1-\delta)^{-\nu k/2} \int_{\sqrt{1-\delta}X_2} e^{-q(t)} dt. \]
We use the bound
\[\left(1-\frac{20 \nu^6 2^{\nu}}{r^2\omega^{\nu}} \ln(k) k^{-\nu+2} \right)^{-\nu k/2} \leq \exp\left(\frac{40 \nu^5 2^{\nu}}{r^2\omega^{\nu}} \ln(k) k^{-\nu+3} \right), \]
and by Lemma ~\ref{lemma:MainVariance} and the choice of the lower bound in the definition of $X_2$, we get
\begin{equation}\label{eq:MiddleRegionNegligibleBinaryPoint} \left|\int_{X_2} F(t) dt\right| \leq \nu k  \exp\left(\frac{40 \nu^5 2^{\nu}}{r^2\omega^{\nu}} \ln(k) k^{-\nu+3}-\frac{5 \omega^{5\nu-3}r 2^{\nu}}{2}  k^{.25} \right) \int_{\mathcal{L}}e^{-q(t)} dt. \end{equation}

Similarly, by Lemma ~\ref{lemma:MainVariance}, we get
\begin{equation}\label{eq:OutsideGaussianNegligibleBinaryPoint} \int_{X_2\cup X_3} e^{-q(t)}dt \leq \nu k \exp\left(-5 \omega^{5\nu-3}r 2^{\nu} k^{.25} \right)\int_{\mathcal{L}} e^{-q(t)} dt.\end{equation}
We define
\[ X_1 = \left\{t\in \mathcal{L}\ :\ ||t||_{\infty}^2 \leq \frac{10 \nu^4 2^{\nu}}{r \omega^{\nu}} k^{-\nu+1.25} \right\} \]
For $t\in X_1$, the inequality $\left|\left<a_j, t \right>\right|^4 \leq \nu^4 ||t||_{\infty}^4$ gives us
\[ |h(t)| \leq \frac{200  \nu^{12} 4^{\nu}}{r^2 \omega^{2\nu}}k^{-\nu+2.5}. \]
Hence, writing
\[ \left|\int_{X_1} F(t) - e^{-q(t)} dt \right| =  \int_{X_1}\left| e^{-q(t)+if(t) + h(t)} - e^{-q(t)}\right| dt, \]
we get
\[\left|\int_{X_1} F(t) - e^{-q(t)} dt \right| \leq \left(\exp\left(\frac{200   \nu^{12}4^{\nu}}{r^2 \omega^{2\nu}}k^{-\nu+2.5}\right)-1\right) \left|\int_{X_1} e^{-q(t)+if(t)} - e^{-q(t)} dt \right|.\]
Applying Lemma ~\ref{lem:InsideRegion} with $\left|\beta^3_{m_1,\ldots,m_{\nu}}\right| = \left|(\zeta_{m_1,\ldots,m_{\nu}}-\zeta_{m_1,\ldots,m_{\nu}}^2)(2\zeta_{m_1,\ldots,m_{\nu}}-1)\right| \leq 1/2$, and noting that almost all the measure of $e^{-q(t)}$ is contained in $X_1$ by ~\eqref{eq:OutsideGaussianNegligibleBinaryPoint}, we get
\begin{equation}\label{eq:InsideRegionBinaryPoint}\left|\int_{X_1} F(t) - e^{-q(t)} dt\right| \leq\left(\exp\left(\frac{200 \nu^{12} 4^{\nu}}{r^2 \omega^{2\nu}}k^{-\nu+2.5}\right)-1\right)\times \end{equation}
\[ \left(2 \nu k \exp\left(-5 \omega^{5\nu-3}r2^{\nu} k^{.25} \right) + 1+\frac{840\nu^{13} }{\omega^{21\nu-15}r^9}k^{2-\nu}\right)\int_{\mathcal{L}} e^{-q(t)} dt. \]
Combining Equations ~\eqref{eq:X3BinaryPoints},~\eqref{eq:MiddleRegionNegligibleBinaryPoint},~\eqref{eq:OutsideGaussianNegligibleBinaryPoint}, and ~\eqref{eq:InsideRegionBinaryPoint} completes the proof.  If $k$ is large enough, the $k^{-2.5+\nu}$ term from ~\eqref{eq:InsideRegionBinaryPoint} dominates, and doubling it gives us the example value for $\Gamma$. 
\end{proof}

\section*{Acknowledgements} This research was partially supported by NSF Grants DMS 0856640 and NSF Grant DMS 1361541.






\end{document}